\documentclass[oneside]{amsart}
\usepackage[all]{xy}

\usepackage[T1]{fontenc}
\usepackage{amssymb,latexsym,amsmath,amscd,amsthm}
\usepackage{xspace}
\usepackage{enumerate}
\usepackage{graphicx}
\usepackage{vmargin}

\usepackage{tikz}
\usetikzlibrary{calc,positioning}

\usepackage[normalem]{ulem}

\DeclareMathOperator{\red}{red}
\DeclareMathOperator{\Res}{Res}

\theoremstyle{plain}
\newtheorem{theorem}{Theorem}[section]
\newtheorem{proposition}[theorem]{Proposition}
\newtheorem{corollary}[theorem]{Corollary}
\newtheorem{lemma}[theorem]{Lemma}

\newtheorem{question}[theorem]{Question}
\newtheorem*{claim*}{Claim}

\theoremstyle{definition}

\newtheorem{remark}[theorem]{Remark}
\newtheorem{example}[theorem]{Example}

\theoremstyle{plain}
\newcounter{claimCount}
\newtheorem{claim}[claimCount]{Claim}

\newcommand{\enm}[1]{\ensuremath{#1}}          %

\newcommand{\cal}[1]{\mathcal{#1}}

\renewcommand{\bar}[1]{\overline{#1}}

\newcommand{\bbA}{\mathbb{A}}
\newcommand{\bbC}{\mathbb{C}}
\newcommand{\bbF}{\mathbb{F}}
\newcommand{\bbN}{\mathbb{N}}
\newcommand{\bbP}{\mathbb{P}}

\newcommand{\LT}{\mathrm{LT}}
\newcommand{\HP}{\mathrm{HP}}

\newcommand{\calA}{\mathcal{A}}
\newcommand{\calI}{\mathcal{I}}
\newcommand{\calL}{\mathcal{L}}
\newcommand{\calO}{\mathcal{O}}
\newcommand{\calS}{\mathcal{S}}
\newcommand{\calZ}{\mathcal{Z}}

\newcommand{\frakS}{\mathfrak{S}}

\newcommand{\CC}{\enm{\mathbb{C}}}

\newcommand{\PP}{\enm{\mathbb{P}}}

\newcommand{\Ii}{\enm{\cal{I}}}

\newcommand{\Oo}{\enm{\cal{O}}}

\newcommand{\Ss}{\enm{\cal{S}}}

\newcommand{\Zz}{\enm{\cal{Z}}}

\renewcommand{\phi}{\varphi}
\renewcommand{\theta}{\vartheta}
\renewcommand{\epsilon}{\varepsilon}

\newcommand{\id}{\mathrm{id}}

\newcommand{\dotitem}{\item[$\cdot$]}
\newcommand{\vvirg}{, \dots ,}
\newcommand{\ootimes}{ \otimes \cdots \otimes}
\newcommand{\ttimes}{ \times \cdots \times}

\newcommand{\rank}{\mathrm{rank}}

\newcommand{\uR}{\underline{R}}
\DeclareMathOperator{\Hom}{Hom}
\usepackage{mathrsfs} 
\newcommand{\scrV}{\mathscr{V}}
\DeclareMathOperator{\Sym}{Sym}

\newcommand{\xto}[1]{\xrightarrow{#1}}

\newcommand{\uc}{\underline{c}}

\newcommand{\textsum}{\textstyle \sum}

\usepackage[font=small,labelfont=bf]{caption}
\usepackage{pgfplots}
\pgfplotsset{compat=1.8}

\usepackage[foot]{amsaddr}

\title[On the partially symmetric rank of tensor products]{On the partially symmetric rank of tensor products of W-states and other symmetric tensors}

\author[E. Ballico]{Edoardo Ballico}
\address[E. Ballico]{Dipartimento di Matematica, Univ. Trento, Italy, \normalfont \lowercase{\texttt{edoardo.ballico@unitn.it}}}

\author[A. Bernardi]{Alessandra Bernardi}
\address[A. Bernardi]{Dipartimento di Matematica, Univ. Trento, Italy, \normalfont \lowercase{\texttt{alessandra.bernardi@unitn.it}}}

\author[M. Christandl]{Matthias Christandl}
\address[M. Christandl]{QMATH - Univ. of Copenhagen, Denmark, \normalfont \lowercase{\texttt{christandl@math.ku.dk}}}

\author[F. Gesmundo]{Fulvio Gesmundo}
\address[F. Gesmundo]{QMATH - Univ. of Copenhagen, Denmark, \normalfont \lowercase{\texttt{fulges@math.ku.dk}} (corresponding author)}

\keywords{Partially symmetric rank, cactus rank, tensor rank, $W$-state, entanglement}
\subjclass[2010]{15A69; 14M20, 14N05}

\begin{document}

\begin{abstract}
Given tensors $T$ and $T'$ of order $k$ and $k'$ respectively, the tensor product $T \otimes T'$ is a tensor of order $k+k'$. It was recently shown that the tensor rank can be strictly submultiplicative under this operation ([Christandl-Jensen-Zuiddam]). We study this phenomenon for symmetric tensors where additional techniques from algebraic geometry are available. The tensor product of symmetric tensors results in a partially symmetric tensor and our results amount to bounds on the partially symmetric rank. Following motivations from algebraic complexity theory and quantum information theory, we focus on the so-called \emph{$W$-states}, namely monomials of the form $x^{d-1}y$, and on products of such. In particular, we prove that the  partially symmetric rank of $x^{d_1 -1}y \ootimes x^{d_k-1} y$ is at most $2^{k-1}(d_1+ \cdots +d_k)$.
\end{abstract}

\maketitle

\section{Introduction}\label{section: intro}

We write $S^d \bbC^2$ for the subspace of symmetric tensors in $(\bbC^2)^{\otimes d}$ and we identify it with the space of complex homogeneous polynomials of degree $d$ in two variables. Given a partially symmetric tensor $T\in S^{d_1}\mathbb{C}^2\otimes \cdots \otimes S^{d_k}\mathbb{C}^2$, a structured tensor decomposition of $T$ is a decomposition of $T$ as 
\begin{equation}\label{tensor:decomposition}
T=\sum_{i=1}^r v_{i,1}^{\otimes d_1}\otimes \cdots \otimes v_{i,k}^{\otimes d_k},
\end{equation}
with $v_{i,j} \in \bbC^2$. The minimum integer $r$ for which an expression as in \eqref{tensor:decomposition} exists is the \emph{partially symmetric rank} of $T$, that we denote by $R_{d_1 \vvirg d_k}(T)$.

In this paper, we focus on the submultiplicativity of the partially symmetric rank: if $T_1 \in S^{d_1} \bbC^2 \ootimes S^{d_i} \bbC^2$ and $T_2 \in S^{d_{i+1}} \bbC^2 \ootimes S^{d_k} \bbC^2$, then it is clear that $R_{d_1 \vvirg d_k}(T_1 \otimes T_2) \leq R_{d_1 \vvirg d_i}(T_1)\cdot R_{d_{i+1} \vvirg d_k}(T_2)$. It has recently been shown in \cite{cjz} that this inequality can be strict; in this paper we further investigate this strict multiplicativity.

We focus specifically on the tensor $W_d \in S^d \bbC^2 \subseteq (\bbC^2)^{\otimes d}$, which is called $W$-state in the physics literature, and is defined as 
\[
W_d = y \otimes x \ootimes x + x \otimes y \otimes x \ootimes x + \cdots + x \ootimes x \otimes y,
\]
where $\{x,y\}$ is a basis of $\bbC^2$; as a homogeneous polynomial in $x$ and $y$, we have $W_d = x^{d-1}y$; it is known that $R_d(W_d) = R_{1\vvirg 1}(W_d) = d$. The proof that $W_3 \otimes W_3$ has rank less than or equal to $8$ is one of the simplest examples of strict multiplicativity of tensor rank. The general techniques of \cite{cjz} also provide a $O(k 2^k)$ upper bound for the tensor product of $k$ copies of $W_3$. The upper bound of $8$ was later shown to be tight in \cite{cf}, where the upper bound for multiple copies was also improved for values of $k$ up to $9$. With a focus on partially symmetric rank and advanced tools from algebraic geometry, we improve upon these bounds and provide a number other insights on the rank of tensor products of symmetric tensors.

\subsection{Motivations}
Tensor decomposition for structured tensors is a classical topic that has been studied in algebraic geometry at least since the nineteenth century and finds numerous applications in other fields, such as quantum physics and theoretical computer science. We present some of the applications in related fields.

\emph{\underline{Entanglement}.} The Hilbert space of a composite quantum system is the tensor product of the Hilbert spaces of the constituent systems. The Hilbert space of the $N$-body system is obtained as the tensor product of $N$ copies of the $n$-dimensional single particle Hilbert space $\mathcal{H}_1$. In the case of indistinguishable bosonic particles, the totally symmetric states under particle exchange are physically relevant, which amounts to restricting the attention to the subspace $\mathcal{H}_s=S^N \mathcal{H}_1\subset \bigotimes^N \mathcal{H}_1$ of completely symmetric tensors. In case we have two different species of indistinguishable bosonic particles, the relevant Hilbert space is $S^{N_1} \mathcal{H}_1\otimes S^{N_2} \mathcal{H}_2$. Tensor rank is a natural measure of the entanglement of the corresponding quantum state (\cite{YCGD10}, \cite{BernCaru:AlgGeomToolsEntanglement}) and strict submultiplicativity of partially symmetric rank reflects the unexpected fact that entanglement does not simply ``add up'' in the composite system formed by multiple bosonic systems, even if the states $T \in S^{N_1} \mathcal{H}_1\otimes S^{N_2} \mathcal{H}_2$ of the two species is a tensor product $T=T_1 \otimes T_2$, where $T_i \in S^{N_i} \mathcal{H}_i$. The results of this paper expand on this novel quantum effect.

\emph{\underline{Communication Complexity}.} The log-rank of the communication matrix is a lower bound on the deterministic communication complexity (see \cite{MehSch:LasVegasBetter}) and it is an open question whether this bound is tight up to polynomial factors (\cite{LovSaks:LatticesMoebFuncsCommCompl}). Recently, it has been shown that \emph{support tensor rank} equals the non-deterministic multiparty quantum communication complexity in the quantum broadcast model (\cite{BuhChrZui:NondetQCC_CyclicEqGameIMM}). Here, the communicating parties obtain each an input and are asked to compute a Boolean function of the joint input using as little quantum communication as possible. The tensor encodes the Boolean function; the order of the tensor corresponds to the number of parties. Support tensor rank is upper bounded by tensor rank with equality in some cases: for instance, in the case of $W$-states or asymptotically in the equality problem, as a consequence of \cite{CohUma:FastMatMultCohConf}. Playing the game independently in two groups of parties but requiring both games to be won corresponds to the tensor rank of the tensor product of the functions. Strict submultiplicativity shows that one can get a reduction in the communication complexity when the two games are played with a joint strategy. 

\emph{\underline{Algebraic Complexity Theory}.} Tensors in $(\bbC^n)^{\otimes 3}$ encode bilinear operations; Strassen showed that the computational complexity of the bilinear map associated to the tensor $T$ is closely related to the tensor rank (\cite{Strassen:RankOptimalComputationGenericTensors}) and asymptotically it is related to the so-called asymptotic rank $\uwave{R}(T):=\lim_{n \rightarrow \infty }R(T^{\boxtimes n})^{1/n}$, where $T^{\boxtimes k}$ denotes the Kronecker product (or flattened tensor product) of tensors, where the tensor power $T^{\otimes k}$ is regarded as an element of $((\bbC^n)^{\otimes k})^{\otimes 3}$. One is interested in studying the gap between $\uwave{R}(T)$ and $R(T)$: this gap can arise both from the fact that $R(T^{\otimes k})$ can be strictly smaller than $R(T)^k$ (that is strict submultiplicativity) and from the fact that $R(T^{\boxtimes k})$ can be strictly smaller than $R(T^{\otimes k})$ (namely passing to the flattened tensor product). This phenomenon has been studied in \cite{cjz} in the context of submultiplicativity of tensor rank and in \cite{ChrGesJen:BorderRankNonMult} in the context of submultiplicativity of border rank. We believe that better understanding of strict submultiplicativity can lead to useful insights on the asymptotic rank. Moreover $W$-states play an important role in the study of the complexity of matrix multiplication: indeed $W_3$ is ``the outer structure'' of the Coppersmith-Winograd tensor (see \cite{CopperWinog:MatrixMultiplicationArithmeticProgressions}, \cite{BerDalHauMou:TensorDecompHomCont}) on which the most recent results concerning upper bounds on the exponent of matrix multiplication are based (see \cite{Stoth:ComplexityMatrixMultiplication}, \cite{Williams:MultMatricesFasterCW}, \cite{LeGall:PowersTensorsFastMatrixMult})

\emph{\underline{W-states.}} Besides what is mentioned above, $W$-states are of key importance both in algebraic geometry and quantum information theory. We mention that tensors of type $W$ are the simplest examples showing that tensor rank fails to be upper semicontinuous. The study of this phenomenon has a long history: it was known to geometers in the 19th century and was then rediscovered in the 80s (see e.g. \cite{BinLotRom:ApproxSolutionsBilFormCompProb}) when it motivated the introduction of the notion of \emph{border rank}. From the point of view of quantum information theory, the $W$-state $W_3 \in S^3 \bbC^2 \subseteq (\bbC^2)^{\otimes 3}$ is one of the two genuinely multiparty entangled classes, the other one being the so-called GHZ type represented by the cubic $x^3 + y^3$ (see \cite{MR1804183}). If the number of particles is higher than three the situation is more complicated (see Remark \ref{remark: sylvester thm}), and generalizations of the $W$-states (so-called Dicke states, corresponding to the monomials in two variables) play a role as well (see e.g. \cite{HolLuqThi:GeomDescrEntangAuxVars}). Asymptotic entanglement distillation properties of the $W$-states have been investigated in \cite{VraChr:EntanglDistillGHZ}. A random distillation protocol to obtain a maximally entangled pairs from $W_d$ by local operations and classical communication (LOCC) is presented in \cite{FortLo:RandomBipartiteEntFromW,FortLo:RandomPartEntDistMultipartyStates}.

\subsection{Main contributions and structure of the paper}
In Section \ref{section: Preliminaries}, we provide preliminary results that will be useful in the rest of the paper. In Section \ref{section: upper bounds}, we prove several results on the upper bounds of partially symmetric rank, in general (\ref{subsec: generic upper bounds}), for products of $W$-states (\ref{subsec: upper bounds W}) and for other special tensors (\ref{subsec: upper bounds others}). Thm. 3.3  improves the bound of Prop. 13 in \cite{cjz} by roughly a factor of $4$; Eqn. \eqref{eqn: R333 expression for W3 tatata W3} in the proof of Thm. 3.3 answers the problem raised in Open Problems 16.5 of \cite{cf}. Similarly, Thm. \ref{bo1} (and in particular Cor. \ref{corol: bound for maximal rank tensors}) improves the bound of Prop. 13 in \cite{cjz} by roughly a factor of $2$. In Section \ref{section: lower bounds}, we give lower bounds on the partially symmetric rank. The bound of Prop. \ref{lower1} compares to the one of Thm. 11 in \cite{Zuid:NoteGapRankBorderRank}, which in turn, when the $d_j$ are not all the same, applies with $d = \min_j \{ d_j\}$ and gives $R_{d_1\vvirg d_k}(W_{d_1} \ootimes W_{d_k}) \geq 2^k (d-1) -d +2$ as in \cite{CheChiDuaJiWin:TensorRankStocEntCatalysisMultPureStates}: for every given values of the parameters, it is straightforward to verify which of the two bounds is better but it is not easy to provide exact conditions; for instance, when $k = 2$, and $d_2 \geq 2d_1+2$, then the bound of Prop. \ref{lower1} improves the one of \cite{CheChiDuaJiWin:TensorRankStocEntCatalysisMultPureStates}; more in general we can observe that the bound of this paper is better when few of the $d_j$'s are much larger than the others. Prop. \ref{prop: multiplicative with W2} partially answers Open Problems 16.1 of \cite{cf}.  Section \ref{section: uniqueness} is dedicated to results on the set of rank one tensors, and more generally on the zero-dimensional scheme supported at a set of rank $1$ tensors, that spans a given partially symmetric tensor. Thm. \ref{thm: Z unique} in the case $k \geq 2$ is original to the extent of our knowledge. Finally, the Appendix (Section \ref{appendix}) contains a brief discussion on the classical Sylvester's Theorem for binary forms (\cite{Sylvester1852}) which inspires most of the techniques used in the rest of the paper, some results about flattening techniques, which are useful tools for lower bounds on several notions of rank and an example giving some insight on the subtleties of zero-dimensional schemes minimally spanning a point.

\subsection*{Acknowledgments} The present paper was conceived during the International workshop on ``Quantum Physics and Geometry'' (Trento, Italy, July 2017) for which we thank CIRM, GNSAGA of INDAM, Mathematical Department of University of Trento, BEC Center, Siquro Project, AQS Project, INFN and TIFPA for their financial support. E.B. is partially supported by GNSAGA of INDAM (Italy) and MIUR PRIN 2015 ``Geometria delle variet\`a algebriche''. A.B. acknowledges support from MIUR PRIN 2015 ``Geometria delle variet\`a algebriche''. M.C. and F.G. acknowledge financial support from the European Research Council (ERC Grant Agreement no. 337603), the Danish Council for Independent Research (Sapere Aude), and VILLUM FONDEN via the QMATH Centre of Excellence (Grant no. 10059). We thank the anonymous referee for useful comments and suggestions.

\section{Notation and Preliminaries}\label{section: Preliminaries} If $V$ is a vector space, $\bbP V$ denotes the projective space of lines in $V$; if $v \in V$, we denote by $[v]$ the corresponding point in $\bbP V$. If $V = \bbC^{n+1}$, we write $\bbP^n = \bbP \bbC^{n+1}$. We often identify $\bbC^2$ with the space of complex linear forms in two variables; in this case we endow $\bbC^2$ with a basis $\{x,y\}$ and $\bbC^{2*}$ with a dual basis $\{ \partial_x,\partial_y\}$. If $X \subseteq \bbP^n$ is a projective variety (or a scheme), we denote by $I_X \subseteq \Sym (\bbC^{n+1 *})$ its homogeneous ideal, where $\Sym(\bbC^{n+1 *})$ denotes the algebra of polynomials on $\bbC^{n+1}$. The span of a variety (or a scheme) $X$ is the variety cut out by the homogeneous component of degree $1$ in $I_X$, namely $(I_X)_1$; it is a projective subspace of $\bbP V$ and, in fact, it is the smallest projective subspace of $\bbP V$ containing $X$.

We refer to \cite{EisHar:GeometrySchemes} for basics on zero-dimensional schemes. Informally, a zero-dimensional scheme can be thought as a set of distinct points each of which has a multiplicity structure arising from the intersection degrees of the hypersurfaces cutting out the point locally. In general, a zero-dimensional scheme is described by the ideal that cuts it out. For instance, on $\bbP^1$, we describe the zero-dimensional scheme $Z$ supported at $[x]$ with multiplicity $2$ by saying that it is the scheme cut out by the ideal $I_Z = ( \partial_y^2)$. Similarly, setting $\bbC^3 = \langle x,y,z\rangle$, the ideal $(\partial_x, \partial_y)^2 \subseteq \bbC[\partial_x,\partial_y,\partial_z]$ cuts out a zero-dimensional scheme $B$ supported at $[z]$ which can be pictured as a point such that the intersection with every line through it is the zero-dimensional scheme $Z$ of degree $2$ described above; $B$ is a zero-dimensional scheme of degree $3$ and following A.V. Geramita it is usually referred to as \emph{fat point} (see \cite{Gera:InvSysFatPts}). If $A \subseteq \bbP^n$ is a zero-dimensional scheme, we say that $A$ is \emph{linearly independent} if $\dim \langle A \rangle = \deg(A) - 1$. Here the dimension is projective.

Given a nondegenerate variety $X \subseteq \bbP^N$ and a point $p\in \bbP^N$, we define the $X$-rank of $p$, denoted $R_X(p)$, to be the minimum $r$ such that $p \in \sigma_r^\circ(X) := \bigcup_{q_1 \vvirg q_r}\langle q_1 \vvirg q_r\rangle$. The $X$-border rank of $p$, denoted $\underline{R}_X(p)$, is the minimum $r$ such that $p$ is the limit of points of $X$-rank $r$, or equivalently $p \in \sigma_r(X) = \overline{\sigma_r^\circ(X)}$, where the overline denotes the Euclidean (or equivalently Zariski) closure.

For $d_1 \vvirg d_k \in \bbN$,  the map
\begin{align*}
 \nu_{d_1 \vvirg d_k} : (\bbP ^{1})^k  &\to  \mathbb{P}(S^{d_1}\mathbb{C}^{2}\otimes \cdots \otimes S^{d_k}\mathbb{C}^{2}   ) \\                                                                                                                                  
 ([v_1] \vvirg [v_k] ) &\mapsto [v_1^{\otimes d_1} \ootimes v_k^{\otimes d_k}],
 \end{align*}
is called the \emph{Segre-Veronese} embedding of $k$ copies of $\mathbb{P}^1$'s in multidegree $(d_1 \vvirg d_k)$. The image of $\nu_{d_1 \vvirg d_k}$ is an algebraic variety, called the Segre-Veronese variety of multidegree $(d_1 \vvirg d_k)$, denoted by $\scrV_{d_1 \vvirg d_k}$. If $k = 1$, $\scrV_{d_1}$ is the $d_1$-th rational normal curve. If $d_1 = \cdots = d_k = 1$, $\scrV_{1 \vvirg 1}$ is the Segre variety of rank $1$ tensors of format $(2 \vvirg 2)$. 

In this setting, the partially symmetric rank of $T$ defined in Section \ref{section: intro} is the $X$-rank where $X = \nu_{d_1 \vvirg d_k} ( (\bbP^1)^k)$. For $T \in S^{d_1} \bbC^2 \ootimes S^{d_k} \bbC^2$, we denote by $R_{d_1 \vvirg d_k}(T)$ (resp. $\underline{R}_{d_1 \vvirg d_k}(T)$) the partially symmetric rank (resp. partially symmetric border rank) of $T$.

We will extensively use the following notion of rank (see e.g. (\cite{ransch,br, BuczBucz:SecantsVeroneseCataSmoothGorSch}). The \emph{$X$-cactus rank} of $p$ is the minimum integer $r$ such that there exists a zero-dimensional scheme $Z \subseteq X$ of degree $r$ with $p \in \langle Z \rangle$ (particular care should be taken if one works with singular varieties -- we will only deal with cases where $X$ is a smooth variety). In this case, we write $c_{X}(p) = r$. Clearly $c_X(p) \leq R_X(p)$. If $X = \nu_{d_1 \vvirg d_k} ( ( \bbP^1)^k)$, we write $c_{d_1 \vvirg d_k} = c_X$.

For a variety $X \subseteq \bbP^n$ and a point $p \in \bbP^n$, we say that a zero-dimensional scheme (resp. a set of distinct points) $A \subseteq X$ \emph{evinces} or \emph{computes} the $X$-cactus rank (resp. the $X$-rank) of $p$ if $\deg(A) = c_X(p)$ (resp. $\deg(A) = r_X(p)$) and $p \in \langle A \rangle$. If $X$ is the image of an embedding $\tilde{X } \to \bbP^n$, we will refer to zero-dimensional schemes in $\tilde{X}$ with the same terminology, referring to the image of the subscheme in the embedding.

We refer to Ch. II and Ch. III in \cite{hart} for an extensive presentation of the theory of sheaf cohomology and its consequences. Given a variety $X$, and a line bundle $\calL$ on $X$, we write $H^k(\calL)$ for the (global) sheaf cohomology groups of $\calL$ and $h^k(\calL)$ for their dimensions. We write $\vert \calL \vert = \bbP (H^0(\calL))$ and we identify it with the space of divisors defined by the sections of $\calL$: in particular, we identify $D \in  \vert \calL \vert$ with the codimension one subscheme defined by its zero locus in $X$. The base locus of $\calL$ is the intersection of the zero loci of all the elements of $\vert \calL \vert$ (\cite{hart}, p.158).

\subsection*{Notation}
Let $X = (\bbP^{1})^k$. Define
 \begin{align*}
  \pi_j :& X \to \bbP^{1}_j \\ 
  \mu_j:& X \to \bbP^{1}_1 \ttimes \bbP^{1}_{j-1} \times \bbP^{1}_{j+1} \ttimes \bbP^{1}_k
 \end{align*}
to be respectively the projection on the $j$-th factor and the projection on all but the $j$-th factor.

The point  $[x] \in \bbP^1$ will be denoted by $o_1$, i.e. $o_1:=[(1,0)]\in \mathbb{P}(\mathbb{C}^2)$ in the basis $\{x,y\}$ of $\bbC^2$. Its defining ideal is $I_{o_1} = (\partial_y) \subseteq Sym(\bbC^{2*}) = \mathbb{C}[\partial_x,\partial_y]$. We denote by $Z_1\subset \mathbb{P}^1$ the zero-dimensional scheme supported at $o_1$ with degree $2$, namely $I_{Z_1} = (\partial_y^2)$. We write $o_k = (o_1 \vvirg o_1) \in (\bbP^1)^k$ and $Z_k = Z_1 \ttimes Z_1 \subseteq (\bbP^1)^k$. We will drop the index $k$ from the notation if it is not essential in the discussion. The double point supported at the point $o$ is denoted by $2o$, which is the zero-dimensional scheme whose ideal is the square of the maximal ideal defining $o$. Notice that the double point $2o_k$ is contained in $Z_k$ but equality only holds for $k=1$. We denote by $L_i \in \vert \calO_{(\bbP^1)^k} ( 0 \vvirg 0, 1,0 \vvirg 0) \vert$ (where $1$ is at the $i$-th entry) the unique divisor with $o \in L_i$; as a subvariety of $(\bbP^1)^k$, we have $L_i = \pi_i^{-1}(o_1)$. We denote by $Q_i := L_i^2 \in \vert \calO_{(\bbP^1)^k}( 0 \vvirg 0 ,2,0\vvirg 0) \vert$, namely $Q_i = \pi_i^{-1} (Z_1)$. We have $Z_k \subseteq Q_i$ and indeed $Z_k = \bigcap_{1 }^k Q_i$.

\subsection{Cactus rank of product of $W$-states}

We briefly recall the following immediate result, which will be used extensively throughout the paper.

First, recall that $W_d \in \langle \nu_d (Z_1) \rangle$. Indeed, using coordinates $\zeta_0 \vvirg \zeta_d$ on $S^d \bbC^2$ ($\zeta_j$ being the coefficient of $x^{d-j}y^j$), the ideal of $\nu_d(Z_1)$ is $I_{\nu_{d}(Z_1)} = (\zeta_1^2,\zeta_2 \vvirg \zeta_d)$; therefore the span of $\nu_d(Z_1)$ is the line cut out by the linear equations in $I_{\nu_d(Z_1)}$, namely $(\zeta_2 \vvirg \zeta_d)$, which is the line parametrized by $\zeta_0,\zeta_1$, $L = \{ \zeta_0 x^d + \zeta_1 W_d\}$; indeed $L$ contains $W_d$. This shows that $c_d(W_d) \leq  2$. Since $W_d \notin \scrV_d$, we have $c_d(W_d) =  2$.

The following result determines the cactus rank of the product of copies of $W$-states, using standard flattening methods (see the Appendix \ref{appendix: flattenings} for details). 

\begin{lemma}\label{lemma: cactus of Wd1 tatata Wdk}
Let $T = W_{d_1} \ootimes W_{d_k} \in S^{d_1} \bbC^2 \ootimes S^{d_k} \bbC^2$ for some nonnegative integers $d_1 \vvirg d_k$. Then $c_{d_1 \vvirg d_k} (T) = 2^k$ and $T \in \langle \nu_{d_1 \vvirg d_k} (Z_k) \rangle$
\end{lemma}
\begin{proof}
Consider the flattening map:
\begin{align*}
T_{1 \vvirg 1} : S^1\bbC^{2*} \ootimes S^1\bbC^{2*} &\to S^{d_1 - 1} \bbC^2 \ootimes S^{d_k -1} \bbC^2 \\
D &\mapsto D(T)
 \end{align*}
where $ S^1\bbC^{2*} \ootimes S^1\bbC^{2*}$ acts naturally on $S^{d_1} \bbC^2 \ootimes S^{d_k} \bbC^2$ by component-wise contraction.  One can verify that this map is injective, namely $\rank (T_{1 \vvirg 1}) \geq 2^k$. It is classically known that the rank of this map gives a lower bound on $c_{d_1 \vvirg d_k}(T)$ (see also \cite{berbramou} and \cite{gal}) providing $c_{d_1 \vvirg d_k}(T) \geq 2^k$.

On the other hand, we have
\[
\langle \nu_{d_1 \vvirg d_k}(Z_k) \rangle = \langle \nu_{d_1 \vvirg d_k} (Z_1 \ttimes Z_1)  \rangle = \langle \nu_{d_1} (Z_1) \rangle \ootimes \langle \nu_{d_k}(Z_1) \rangle;
\]
therefore $T \in \langle \nu_{d_1 \vvirg d_k}(Z_k) \rangle$. Since $\deg(Z_k) = 2^k$, we have $c_{d_1 \vvirg d_k}(T) \leq 2^k$ and we conclude.
\end{proof}

Lemma \ref{lemma: cactus of Wd1 tatata Wdk} shows that $Z_k$ is a minimal zero-dimensional scheme such that $W_{d_1} \ootimes W_{d_k} \in \langle \nu_{d_1 \vvirg d_k} (Z_k) \rangle$. In particular, no proper subscheme $Z \subsetneq Z_k$ satisfies $ W_{d_1} \ootimes W_{d_k} \in \langle \nu_{d_1 \vvirg d_k} (Z) \rangle$. In fact, Theorem \ref{thm: Z unique} will show that $Z_k$ is the unique zero-dimensional scheme evincing the cactus rank of $W_{d_1} \ootimes W_{d_k}$.

\subsection{Two useful exact sequences and their consequences}

This section has the double purpose to state some known results in the language that will be used in the rest of the paper and to introduce some tools and preliminary results that will be useful in the next sections.

Most of the arguments that we will use follow from the study of the long exact sequence in cohomology arising from an exact sequence of sheaves. 

Let $X$ be a variety and $Y \subseteq X$ a subscheme.  We write $\calI_{Y,X}$ for the ideal sheaf of $Y$ in $\calO_{X}$; we write $\calI_Y$ if no confusion arises. Then the following sequence (called the \emph{restriction exact sequence} of $Y$) is exact
\begin{equation}\label{eqn: restriction exact sequence}
0 \to \calI_Y \to \calO_X \to \calO_X \vert _Y \to 0.
\end{equation}
We will use this exact sequence several times, often tensoring it with a line bundle $\calL$ on $X$. The restriction map $\sigma^X_Y : H^0(\calL) \to H^0(\calL\vert _Y)$ appears in the resulting long exact sequence in cohomology
\[
 0 \to H^0( \calI_Y \otimes \calL) \to H^0(\calL) \xto{\sigma^X_Y} H^0(\calL\vert_Y) \to H^1(\calI_Y \otimes \calL) \to H^1(\calL) \to \cdots .
\]
We obtain immediately that if $h^1(\calL) = 0$ then $h^1(\calI_Y \otimes \calL) = h^0 (\calL\vert _Y) - \dim (\mathrm{Im} (\sigma^X_Y) )$. In particular, if $Y$ is zero-dimensional, then $h^1(\calI_Y \otimes \calL) = \deg(Y) - \dim (\mathrm{Im} (\sigma^X_Y))$.

This has the following two easy but important consequences (if $Y_1 \subseteq Y_2 \subseteq X$ we write $\calI_{Y_1,Y_2} = \calI_{Y_1} \vert_{Y_2}$): 

\begin{remark}\label{inclusion}
Let $Y_1 \subseteq Y_2  \subseteq X$. We have $\sigma^X_{Y_1} = \sigma^{Y_2}_{Y_1} \circ \sigma^X_{Y_2}$, so $\dim ( \mathrm{Im} (\sigma^X_{Y_1}) ) \leq \dim ( \mathrm{Im} (\sigma^{Y_2}_{Y_1}))$. If $h^1(\calL) = h^1(\calL \vert_{Y_2}) = 0$, we obtain $h^1(\calI_{Y_1} \otimes \calL) =  h^0 (\calL\vert _{Y_1}) - \dim (\mathrm{Im} (\sigma^X_{Y_1}) ) \geq h^0 (\calL\vert _{Y_2}) - \dim (\mathrm{Im} (\sigma^{Y_2}_{Y_1}) )  = h^1(\calI_{Y_1,Y_2} \otimes \calL \vert_{Y_2}) $. In summary $h^1(\calI_{Y_1} \otimes \calL) \geq h^1(\calI_{Y_1,Y_2} \otimes \calL \vert_{Y_2})$.
\end{remark}

\begin{remark}\label{lemma: h1 - h1 vs deg - deg}
Let $A,B$ be zero-dimensional schemes in $X$ with $A \subseteq B$ and let $\calL$ be a line bundle on $X$ with $h^1(\calL) = 0$. Then
 \[
  0 \leq h^1 (\calI_B \otimes \calL) - h^1 (\calI_A \otimes \calL) \leq \deg(B) - \deg(A).
 \]
\end{remark}

Let $X$ be a variety, $A \subseteq X$ a zero-dimensional scheme and $D \subseteq X$ an effective Cartier divisor. The following sequence (called the \emph{residual exact sequence} of $A$ with respect to $D$ in $X$) is exact:
\begin{equation}\label{eqn: residual exact sequence general}
 0 \to \calI_{\Res_D(A)} \otimes \calI_D \to \calI_A \to \calI_{D \cap A, D} \to 0.
\end{equation}
Here $\Res_D(A)$ is the \emph{residue scheme} of $A$ with respect to $D$, namely the subscheme of $X$ whose ideal sheaf is $\calI_A : \calI_D$. By definition,
\begin{equation}\label{eqn: degree as residue plus intersection}
 \deg(A) = \deg( A \cap D) + \deg(\Res_D(A)).
\end{equation}
Moreover, it is immediate that if $A,B$ are two zero-dimensional schemes, then $\Res_D(A \cup B) = \Res_D(A) \cup \Res_D(B)$. Figure \ref{figure: residues} represents an example of zero-dimensional scheme $A$, with a divisor $D$ on a plane.

\begin{figure}[!htp]
\begin{tikzpicture}
    \node[anchor=south west,inner sep=0] at (0,0) {\includegraphics[width=.7\textwidth]{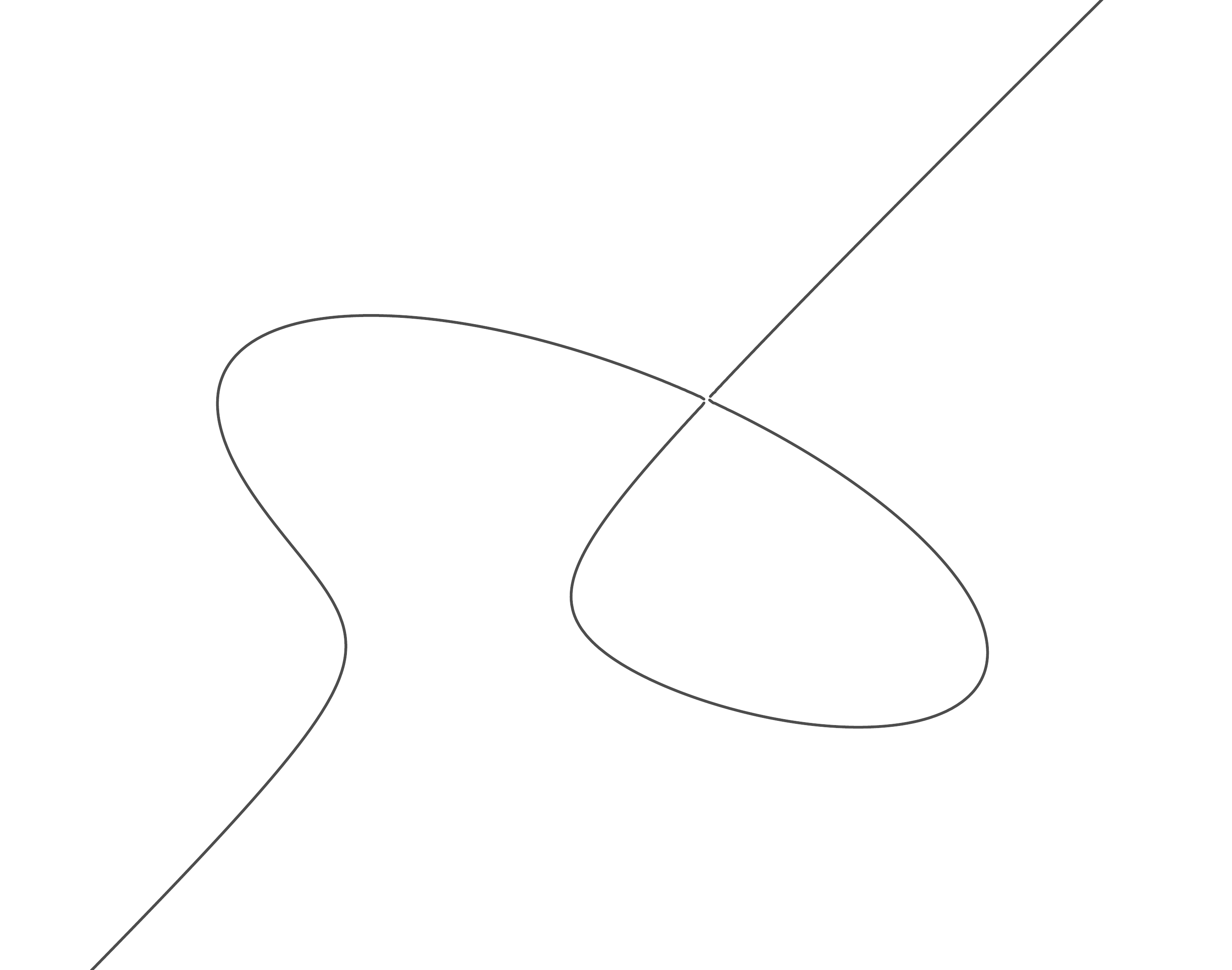}};
    \node at (6.05,4.85)  (p1)  {};
    \node at (2.91,3)  (p2)  {};
    \node at (8.05,6.93) (p3)  {};
    \node at (3,7)  (p4)  {};
    \node at  (7,3.5)  (p5)  {};
    \draw[blue,fill=blue] (p1) circle (2pt);
    \draw[blue,fill=blue] (p2) circle (2pt);
    \draw[red,fill=red] (p3) circle (1.5pt);
    \draw[blue,fill=blue] (p4) circle (2pt);
    \draw[red,fill=red] (p5) circle (1.5pt);
    \draw[dashed] (p1) -- ($ (p1) + (3,0)$);
    \draw[dashed] (p2) -- ($ (p2) - (2,1)$);
    \draw[dashed] (p3) -- ($ (p3) - (1.7,0)$);
    \draw[dashed] (p4) -- ($ (p4) + (1.5,0)$);
    \draw[dashed] (p5) -- ($ (p5) - (0,2.5)$);
\node[right=3pt of {$ (p1) + (3,0)$}] {$p_2$};
\node[left=1pt of {$ (p2) - (2,1)$}] {$p_1$};
\node[left=2pt of {$ (p3) - (1.5,0)$}] {$p_3$};
\node[right=1pt of {$(p4) + (1.5,0)$}] {$p_4$};
\node[right=0pt of {$(p5) - (0,2.5)$}] {$p_5$};
\end{tikzpicture}
{\scriptsize \caption{The zero-dimensional scheme $A$ consists of the points in the picture: red points are simple, blue points are double; in particular $\deg(A) = 3 + 3+ 1 + 3+ 1 = 11$. The divisor $D$ is represented by the singular curve. $D \cap A$ is the zero-dimensional scheme consisting of a point of multiplicity $2$ at $p_1$ (the double point on the tangent line at $p_1$), the fat double $p_2$ and the point $p_3$. $\Res_D(A)$ is the zero-dimensional scheme consisting of a simple point at $p_1$, the double point at $p_4$ and the simple point $p_5$.}\label{figure: residues}} 
\end{figure}
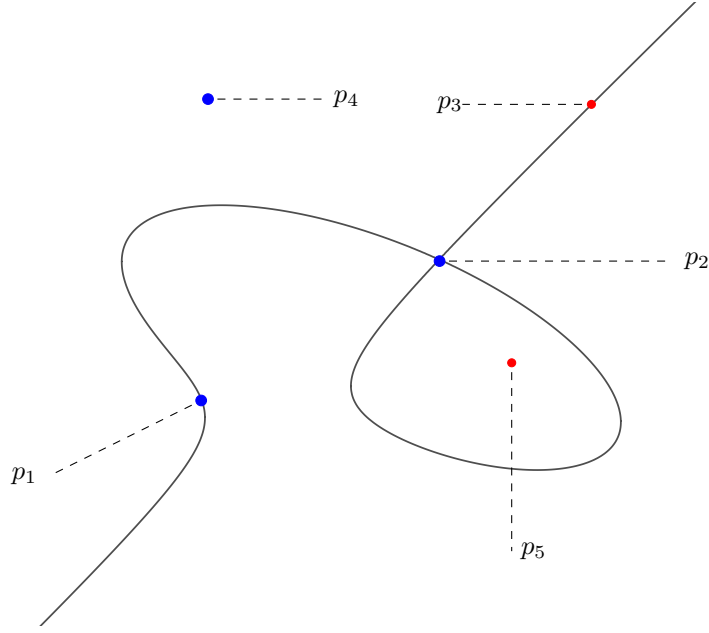

We rephrase the following result into our language.

\begin{lemma}[Lemma 5.1, item (b), \cite{bb2}]\label{lemma: 5.1 in our form}
 Let $X\subset \PP^n$ be an irreducible variety. Let $p\in \PP^n$, let $S \subset X$ be a finite set and $A$ a zero-dimensional scheme such that $p\in \langle A\rangle \cap \langle S\rangle$, $S \neq A$, and $p\notin \langle A'\rangle $ for any $A'\subsetneq A$. Let $D \subset X$ be an effective Cartier divisor. If $h^1(\Ii _{\Res _D(A\cup S)}(1)\otimes \calI_D) =0$. Then $\mathrm{Res}_ D(A) = \Res_D(S)$.
\end{lemma}

We prove a result similar to Lemma \ref{lemma: 5.1 in our form} that will be particularly useful in the next sections.

\begin{lemma}\label{lemma: 5.1 plus}
 Let $X \subseteq \bbP^n$ be an irreducible variety. Let $p\in \PP^n$ and let $A,B$ be zero-dimensional schemes in $X$ such that $p \in \langle A \rangle$, $p \in \langle B \rangle$ and there are no $A' \subsetneq A$ and $B' \subsetneq B$ with $p \in \langle A'\rangle$ or $p \in \langle B' \rangle$. Suppose $h^1(\calI_{B}(1)) = 0$. Let $D \subseteq X$ be an effective Cartier such that $\Res_D(A) \cap \Res_D(B) = \emptyset$. If $h^1 (\calI_{\Res_D(A \cup B)} (1) \otimes \calI_D) = 0$ then $A \cup B \subseteq D$.
\end{lemma}
\begin{proof}
We have $\dim \langle A \rangle = \deg(A) - 1 - h^1 (\calI_A(1))$ and $\dim \langle B \rangle = \deg(B) - 1 $. By Grassmann's formula, we have $\dim (\langle A \rangle \cap \langle B \rangle) = \dim \langle A \rangle + \dim \langle  B \rangle - \dim (\langle A\rangle + \langle B\rangle)$. The last term is $ \dim (\langle A\rangle + \langle B\rangle) = \dim \langle A \cup  B\rangle = \deg (A \cup B) -1 - h^1 (\calI_{A \cup B}(1)) = \deg (A) + \deg(B) - \deg(A \cap B) -1 - h^1 (\calI_{A \cup B}(1)) $. We deduce $\dim (\langle A \rangle \cap \langle B \rangle) = \deg(A \cap B) + h^1 (\calI_{A \cup B}(1)) -1 - h^1 (\calI_A(1))$.

Similarly, we have $\dim (\langle A \cap D \rangle \cap \langle B \cap D\rangle) = \deg(A \cap B \cap D) + h^1 (\calI_{(A \cup B) \cap D}(1)) -1  - h^1 (\calI_{A \cap D}(1))$. Since $\Res _D(B)\cap \Res _D(A) =\emptyset$, we have $A \cap B \cap D = A \cap B$ which provides $\Res_D(A) \cup \Res_D(B) = \Res_D(A \cup B)$.

Consider the residual exact sequence of $A \cup B$ in $X$ with respect to $D$:
 \[
  0 \to \calI_{\Res_D(A \cup B)}(1) \otimes \calI_D \to \calI_{A \cup B}(1) \to \calI_{(A \cup B)\cap D , D} (1) \to 0.
 \]
From the hypothesis, we have $h^1 (\calI_{\Res_D(A \cup B)}(1)\otimes \calI_D) = 0$, so $h^1(\calI_{A \cup B}(1)) \leq h^1(\calI_{(A \cup B)\cap D , D} (1))$. On the other hand, Remark \ref{inclusion} applied to $(A \cup B)\cap D$ provides that $h^1(\Ii _{(A\cup B)\cap D,D}(1)) \leq h^1 (\Ii _{(A\cup B)\cap D}(1)) \leq h^1 (\Ii _{A\cup B}(1))$, so we obtain $h^1(\calI_{A \cup B}(1)) = h^1 (\Ii _{(A\cup B)\cap D}(1))$.

We conclude $\dim (\langle A \rangle \cap \langle B \rangle) - \dim (\langle A \cap D \rangle \cap \langle B \cap D\rangle) =  - h^1 (\calI_A(1)) + h^1 (\calI_{A \cap D}(1)) \leq 0$ and therefore $\langle A \rangle \cap \langle B \rangle = \langle A \cap D \rangle \cap \langle B \cap D\rangle$. This shows $p \in \langle A \cap D \rangle$ and $p \in \langle B \cap D\rangle$ and from the minimality hypothesis we conclude $A \cap D = A$ and $B \cap D = B$ so $A \cup B \subseteq D$.
\end{proof}

We point out that the hypothesis $h^1(\calI_B(1))= 0$ in the hypothesis of Lemma \ref{lemma: 5.1 plus} is necessary and does not follow from the other hypothesis of the lemma. In fact, the condition that $B$ minimally spans $p$ does not guarantee $h^1(\calI_B(1)) = 0$ as shown in Example \ref{example: minimally spanning with h1} in Appendix \ref{section: minimally spanning}.

We will also need the following result. 

\begin{lemma}[Lemma 1, \cite{bb1}]\label{lemma: union minimals have h^1}
Let $p\in \PP^n$ and let $A,B$ be zero-dimensional schemes in $X$ such that $p \in \langle A \rangle$, $p \in \langle B \rangle$ and there are no $A' \subsetneq A$ and $B' \subsetneq B$ with $p \in \langle A'\rangle$ or $p \in \langle B' \rangle$. Then $h^1( \calI_{A \cup B}(1)) > 0$.
\end{lemma}

In Figure \ref{figure:rank W3}, we can schematically observe the effect of Lemma \ref{lemma: union minimals have h^1}: $\nu_3(Z)$ and $S = \{q_1,q_2,q_3\}$ both minimally span $W_3$; the zero-dimensional scheme consisting of the union $Z \cup S$ is not linearly independent and we can verify $h^1(\calI_{Z \cup S} (1)) > 0$ (see also Remark \ref{remark: sylvester thm} and Remark \ref{remark: basics on rhoX}).

\begin{figure}[!htp]
\begin{tikzpicture}
     \node[anchor=south west,inner sep=0] at (0,0) {\includegraphics[width=\textwidth]{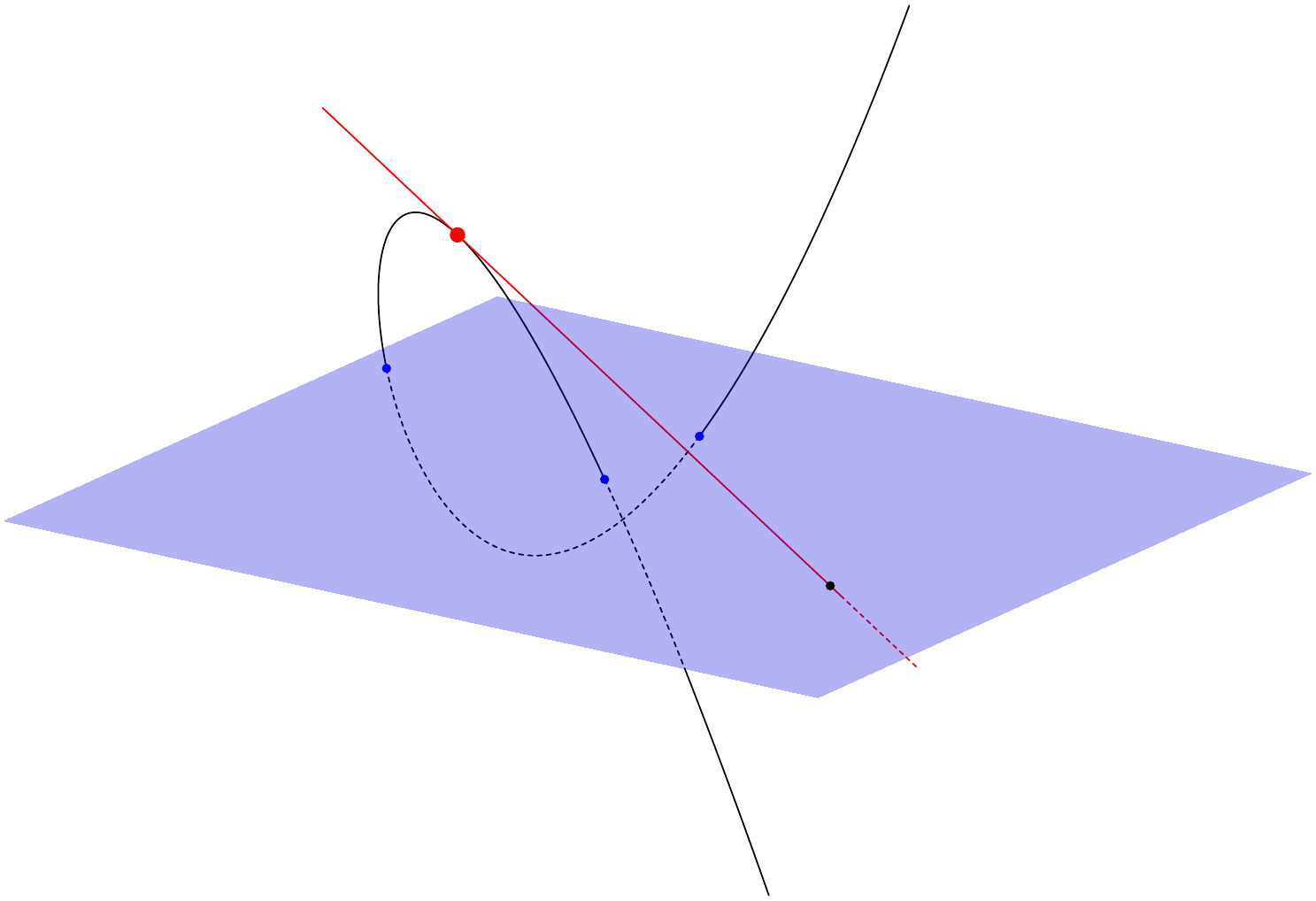}};
    \node at (4.75,6.1)  (q1)  {$q_1$};
    \node at (6.7,4.6)  (q2)  {$q_2$};
    \node at (8.3,5.2)  (q3)  {$q_3$};  
    \node at (9.8,3.7)  (W)  {$W$};
    \node at (5.85,7.7)  (v3)  {$\nu_3(Z)$};
\end{tikzpicture}

\caption{Cactus rank and symmetric rank of $W_3 = x^2y$. The black curve represents $\nu_3(\bbP^1)$. The zero-dimensional scheme $\nu_3(Z)$ has degree $2$ and it is supported at $[x^3] \in \nu_3(\bbP^1)$. The point $W_3$ lies on the span of $\nu_3(Z)$, that is the red line (tangent to $\nu_3(\bbP^1)$), and on the span of the three points $\{q_1,q_2,q_3\}$, that is the blue plane. We have $c_3(W_3) = 2$ and $R_3(W_3) = 3$.}\label{figure:rank W3}
\end{figure}
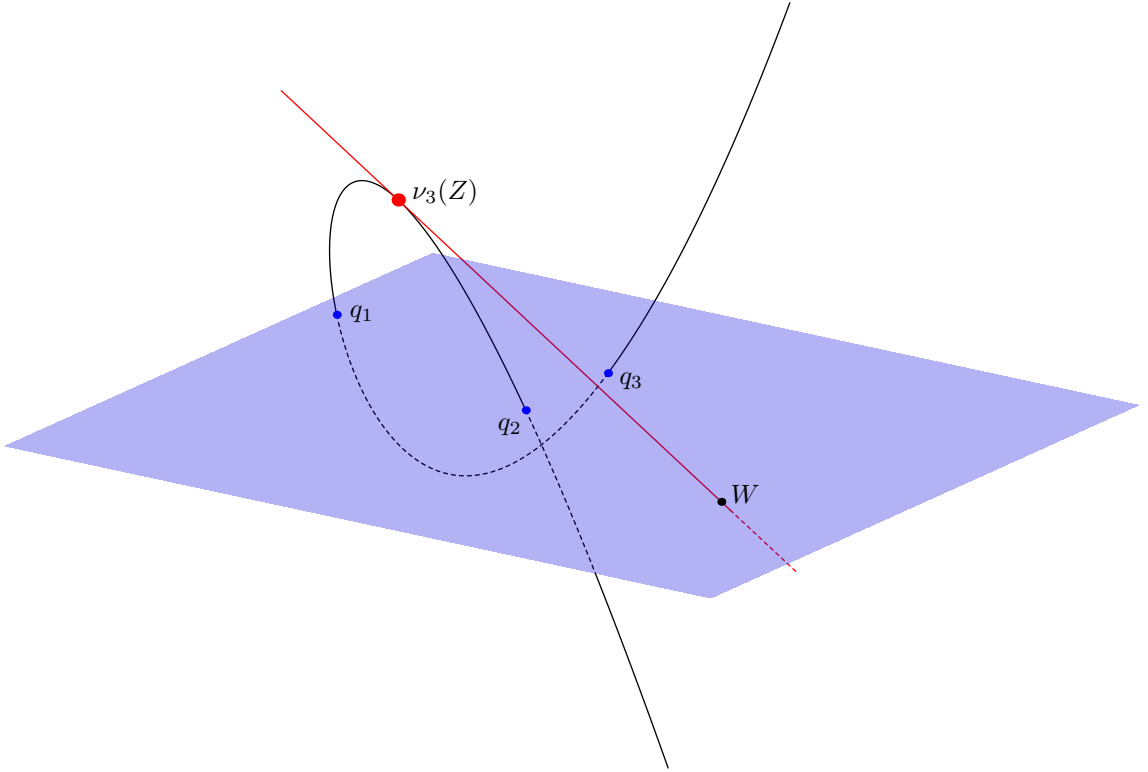

\section{Upper bounds for the partially symmetric rank of tensors}\label{section: upper bounds}

In this section, we provide upper bounds for the partially symmetric rank of certain tensors. Some of the results show that submultiplicativity of rank occurs frequently in this setting. In particular, we exploit upper bounds on the generic rank to obtain upper bounds on the rank showing that submultiplicativity occurs whenever the ranks are significantly larger than the generic. We show stronger upper bounds for the product of $W$-states and for partially symmetric tensors whose factors are bivariate monomials (the so-called Dicke states in the quantum information literature).

\subsection{Bounds via genericity arguments}\label{subsec: generic upper bounds}
In the case of tensors in $S^{d} \bbC^2$ having rank higher than the generic rank, then submultiplicativity is frequent. The reason is the following result that gives an absolute bound on the rank of a partially symmetric tensor.

\begin{proposition}\label{oo2}
 Let $k\geq 2$ and $d_1 \leq \cdots \leq d_k $ be nonnegative integers different from the following
 \begin{itemize}
\dotitem $k=2$, $d_1=2$ and $d_2$ even;
\dotitem $k=3$, $d_1=d_2=1$ and $d_3$ even;
\dotitem $k=3$, $d_1=d_2=d_3=1$;
\dotitem $k=4$, $d_1=d_2=d_3=d_4=1$.
 \end{itemize}
Let $N = \prod_i (d_i + 1)$. Then for every $T \in S^{d_1} \bbC^2 \ootimes S^{d_k} \bbC^2$, we have 
\[
 R_{d_1 \vvirg d_k}(T) \leq 2 \lceil N / (k+1) \rceil.
\]
\end{proposition}
\begin{proof}
Fix $T \in S^{d_1} \bbC^2 \ootimes S^{d_k} \bbC^2$. The list of exceptions for the values of $k$ and $d_i$ guarantees that $\uR_{d_1 \vvirg d_k}(T) \leq \lceil N / (k+1) \rceil$ (see  \cite{cgg},  \cite[Theorem 1.1]{bd}, \cite[Theorem 3.1]{ab3}) because $\lceil N / (k+1) \rceil$ is the generic partially symmetric rank in $\bbP ( S^{d_1} \bbC^2 \ootimes S^{d_k} \bbC^2)$. By \cite{bt}, the maximum rank is bounded from above by twice the generic rank, therefore we conclude.
\end{proof}

Proposition \ref{oo2} implies that if $q_i \in S^{d_i} \bbC^2$ have sufficiently large ranks (for $i = 1 \vvirg k$), then
\[
R_{d_1 \vvirg d_k} ( q_1 \ootimes q_k) < R_{d_1} (q_1) \cdots  R_{d_k} (q_k).
\]
In particular, partially symmetric rank is strictly submultiplicative. More precisely, we have the following corollary:

\begin{corollary}\label{corol: bound for maximal rank tensors}
 Fix integers $k$,  $d_1 \vvirg d_k$ as in Proposition \ref{oo2}. For $i =1 \vvirg k$, let $q_i \in S^{d_i} \bbC^2$ with $R_{d_i} (q_i) = r_i$. If $r_1 \cdots r_k > \frac{2}{k+1} \prod_i (d_i + 1)$, then 
 \[
R_{d_1 \vvirg d_k}(q_1 \ootimes q_k) <  r_1 \cdots r_k.
 \]
\end{corollary}

The hypotheses of Corollary \ref{corol: bound for maximal rank tensors} are satisfied for instance when $r_i = d_i$, namely when $q_i = W_{d_i}$. However, Theorem \ref{bo1} will provide a stronger upper bound for product of $W$-states.

\subsection{Bounds for products of $W$-states}\label{subsec: upper bounds W}

In this section, we provide upper bounds for the partially symmetric rank of the tensor product of $W$-states. We point out that these bounds hold for tensor rank as well. In particular, the following result generalizes the expressions for $W_3 ^{\otimes 2}$ given in \cite{cjz} and for $W_3^{\otimes 3}$ given in \cite{cf} and answers Question 5 in Open Problems 16 of \cite{cf} in the setting of partially symmetric tensors.

\begin{theorem}\label{thm: rank upper bound W3 tatata W3}
 For every $k$, we have $R_{3\vvirg 3} (W_3^{\otimes k}) \leq (2+k) 2^{k-1}$.
\end{theorem}
\begin{proof}
Fix $k$ and use variables $x_i,y_i$, $i=1, \ldots , k$, as basis of the $i$-th copy of $\bbC^2$. In particular $T = W_3^{(1)} \ootimes W_3^{(k)} = x_1^2y_1 \ootimes x_k^2y_k \in S^3 \bbC^2 \ootimes S^3 \bbC^2$. We will prove $R_{3 \vvirg 3} (T)  \leq (2+k)2^{k-1}$.

We determine an expression 
\begin{equation}\label{eqn: R333 expression for W3 tatata W3}
 T = G - \textsum_{i=1}^k H_i,
\end{equation}
where $R_{3 \vvirg 3} (G) \leq  2^k$ and $R_{3 \vvirg 3} (H_i) \leq 2^{k-1}$ for every $i$. Define
\begin{align*}
G = & (W_3^{(1)} + y_1^3) \ootimes (W_3^{(k)} + y_k^3), \\
H_i = & (W_3^{(1)} + a_{i1} y_1^3) \ootimes  (W_3^{(i-1)} + a_{i,i-1}y_{i-1}^3) \otimes y_i^3 \otimes \\
 & \otimes (W_3^{(i+1)} + a_{i,i+1}y_{i+1}^3) \ootimes (W_3^{k} + a_{ik}y_{k}^3),
\end{align*}
where $a_{ij} = \xi_i / (\xi_i-\xi_j)$ for some choice of distinct constants $\xi_i \neq 0,1$.

We claim that with this choice of $a_{ij}$, \eqref{eqn: R333 expression for W3 tatata W3} holds. Indeed, \eqref{eqn: R333 expression for W3 tatata W3} is true if and only if the coefficients $a_{ij}$ satisfy the following set of polynomial equations:
\begin{equation}\label{eqn: conditions on aij}
\left\{ 
\begin{array}{l}
\displaystyle \sum_{p=1}^\ell \Biggl(\prod_{\substack{ q = 1 \vvirg \ell \\ q \neq p } } a_{s_p s_q} \Biggr) -1 =0,\\
\text{for every $ 2\leq \ell \leq k$ and every $s_1 \vvirg s_\ell \in \{ 1 \vvirg k\}$ distinct.}
\end{array}
\right.
\end{equation}
Fix $\ell$ and without loss of generality consider the condition $\sum_{p=1}^\ell \left(\prod_{q \neq p} a_{p q} \right) -1 = 0$. For our choice of $a_{ij}$, the monomial $ \prod_{q \neq p} a_{p q}$ (with a fixed $p$) is
\[
\frac{\xi _p}{(\xi_p - \xi_1)} \cdots \frac{\xi_p}{(\xi_p -\xi_{p-1})} \cdot \frac{\xi_p}{(\xi_p -\xi_{p+1})} \cdots \frac{\xi_p}{(\xi_p -\xi_\ell)};
\]
regarding the $\xi_i$'s as variables, the least common denominator of these monomials is (up to scale) $\prod_{1 \leq \alpha < \beta \leq \ell} (\xi_\alpha - \xi_\beta)$, which has degree $\binom{\ell}{2}$ in the $\xi_j$'s. The $p$-th monomial in the numerator of the expression $\sum_{p=1}^\ell \left(\prod_{q \neq p} a_{p q} \right)$ is 
\begin{equation}\label{eqn: pth summand in numerator of upper bound condition}
(-1)^{p-1} \xi_p^\ell \prod_{\substack{1 \leq \alpha < \beta \leq \ell \\ \alpha,\beta \neq p}} (\xi_\alpha - \xi_\beta).
\end{equation}
\begin{claim*}
 Fix $\bar{\gamma},\bar{\delta}$ and $\ell$. Then the numerator of $\sum_{p=1}^\ell \left(\prod_{q \neq p} a_{p q} \right)$ is divisible by $(\xi_{\bar{\gamma}} - \xi_{\bar{\delta}})$.
\end{claim*}
\begin{quote}
\begin{proof}[Proof of Claim]
Suppose $\bar{\gamma} < \bar{\delta}$. If $p \neq \bar{\gamma},\bar{\delta}$, then the $p$-th summand in the numerator is divisible by $(\xi_{\bar{\gamma}} - \xi_{\bar{\delta}})$ as it appears in the product in \eqref{eqn: pth summand in numerator of upper bound condition}.

From \eqref{eqn: pth summand in numerator of upper bound condition} with $p = \bar{\gamma}$, we obtain that the $\bar{\gamma}$-th summand is
\begin{align*}
 &(-1)^{\bar{\gamma}-1} \xi_{\bar{\gamma}}^\ell \prod_{\substack{1 \leq \alpha < \beta \leq \ell \\ \alpha,\beta \neq \bar{\gamma}}} (\xi_\alpha - \xi_\beta) = \\
 &(-1)^{\bar{\gamma}-1}\xi_{\bar{\gamma}}^\ell  \prod_{\substack{1 \leq \alpha < \beta \leq \ell \\ \alpha,\beta \neq \bar{\gamma}, \bar{\delta}}} (\xi_\alpha - \xi_\beta) \cdot 
 \prod_{1 \leq \alpha \leq \bar{\gamma}-1 } (\xi_{\alpha} - \xi_{\bar{\delta}}) \cdot 
 \prod_{\bar{\gamma}+1 \leq \alpha \leq \bar{\delta}-1 } (\xi_{\alpha} - \xi_{\bar{\delta}}) \cdot 
\prod_{\bar{\delta}+1 \leq \beta \leq \ell } (\xi_{\bar{\delta}} - \xi_\beta).
 \end{align*}
 Similarly from \eqref{eqn: pth summand in numerator of upper bound condition} with $p = \bar{\delta}$, the $\bar{\delta}$-th summand is
 \begin{align*}
(-1)^{\bar{\delta}-1}\xi_{\bar{\delta}}^\ell  \prod_{\substack{1 \leq \alpha < \beta \leq \ell \\ \alpha,\beta \neq \bar{\gamma}, \bar{\delta}}} (\xi_\alpha - \xi_\beta) \cdot 
 \prod_{1 \leq \alpha \leq \bar{\gamma}-1 } (\xi_{\alpha} - \xi_{\bar{\gamma}}) \cdot 
 \prod_{\bar{\gamma}+1 \leq \alpha \leq \bar{\delta}-1 } (\xi_{\bar{\gamma}} - \xi_{\beta}) \cdot 
\prod_{\bar{\delta}+1 \leq \beta \leq \ell } (\xi_{\bar{\gamma}} - \xi_\beta). 
 \end{align*}
Specializing to $\xi_{\bar{\gamma}} = \xi_{\bar{\delta}} = \tilde{\xi}$, we can factor out of the sum of these two terms the product  
\[
\tilde{\xi} ^\ell \prod_{\substack{1 \leq \alpha < \beta \leq \ell \\ \alpha,\beta \neq \bar{\gamma}, \bar{\delta}}} (\xi_\alpha - \xi_\beta) \cdot 
  \prod_{1 \leq \alpha \leq \bar{\gamma}-1 } (\xi_{\alpha} - \tilde{\xi}) \cdot 
  \prod_{\bar{\delta}+1 \leq \beta \leq \ell } (\tilde{\xi} - \xi_\beta)  
\]
obtaining
\begin{align*}
\left( (-1)^{\bar{\gamma}-1} \prod_{\bar{\gamma}+1 \leq \alpha \leq \bar{\delta}-1 } (\xi_{\alpha} - \tilde{\xi})\right) + \left(  (-1)^{\bar{\delta}-1} \prod_{\bar{\gamma}+1 \leq \beta \leq \bar{\delta}-1 } (\tilde{\xi} - \xi_{\beta}) \right) = \\  
\left[ (-1)^{\bar{\gamma}-1} \cdot (-1)^{(\bar{\delta} -1) - (\bar{\gamma}+1) + 1}  + (-1)^{\bar{\delta}-1} \right] \cdot \prod_{\bar{\gamma}+1 \leq \beta \leq \bar{\delta}-1 } (\tilde{\xi} - \xi_{\beta}) = 0.
\end{align*}
This shows that the numerator is divisible by $\xi _ {\bar{\gamma}} - \xi _ {\bar{\delta}}$.
\end{proof}
\end{quote}

Unique factorization implies that the numerator is the same as the denominator up to a constant factor. Checking this constant factor shows that the conditions in \eqref{eqn: conditions on aij} hold.

Recall that if $g \in S^3 \bbC^2$ is a binary cubic with distinct linear factors, then $R_3(g) = 2$ (this is a classical fact due to Sylvester's Theorem, see e.g. \cite[Ex. 3.10]{cgo:four}). This shows that $R_{3 \vvirg 3} (G) \leq 2^k$, because the factors of $G$ are cubic forms with distinct linear factors. Moreover, if $\xi_i \neq 0,1$ for every $i$, then $a_{ij} \neq 0,1$ for every $i,j$ and therefore we have $R_{3 \vvirg 3} (H_i) \leq 2^{k-1}$ for every $i$, because $k-1$ of its factors are cubic forms with distinct linear factors and the $i$-th factor is $y_i^3$ which has rank $1$.

We conclude $R_{3 \vvirg 3}(T) = R_{3 \vvirg 3}(G - \sum_i H_i) \leq 2^k + k 2^{k-1} = (2+k)2^{k-1}$. 
\end{proof}

\begin{remark}
 The argument that is used in Theorem \ref{thm: rank upper bound W3 tatata W3} to write $W^{(1)}_3 \ootimes W^{(k)}_3 = G - \sum_1^k H_i$ works in much higher generality. Fix $d_1 \vvirg d_k$ and let $g_1 \vvirg g_k$ be binary forms with $g_i \in S^{d_i} \bbC^2$. Then, we can write $  W_{d_1} \ootimes W_{d_k} = G - \sum_1 ^d H_i$ where
 \begin{align*}
 G = & (W_{d_1} + g_1) \ootimes (W_{d_k} + g_k) \\
 H_i = &(W_{d_1} + a_{i1} g_1) \ootimes (W_{d_{(i-1)}} + a_{i,i-1}g_{i-1}) \otimes g_i \otimes \\ 
 & \otimes (W_{d_{(i+1)}} + a_{i,i+1}g_{i+1}) \ootimes (W_{d_k} + a_{ik}g_k),
 \end{align*}
and the constants $a_{ij}$ are chosen in the same way as in Theorem \ref{thm: rank upper bound W3 tatata W3}. This argument can be used to obtain an upper bound for $R_{d_1 \vvirg d_k}( W_{d_1} \ootimes W_{d_k})$ but this bound is worse than the bound that we provide in Theorem \ref{bo1}. For this reason, we do not give the details of this construction.
\end{remark}

The next two results provide upper bounds on the partially symmetric rank of products of $W$-states for $d_i \geq 3$. The argument is essentially the same for the two results: it is based on determining a collection of rational normal curves in $\nu_{d_1 \vvirg d_k} ((\bbP^1)^{k})$ intersecting at the point $o$ with high multiplicity, so that their union contains the zero-dimensional scheme $Z$. The upper bound on the rank follows from the upper bound on the dimension of the span of the union of the rational curves.

In order to gain intuition for this procedure, we first prove the upper bound in the case $k=2$. In particular the following result improves the upper bound of Corollary \ref{corol: bound for maximal rank tensors} which was obtained via a genericity argument.

\begin{proposition}\label{prop: upper bound Wd1 ot Wd2}
For all positive integers $d_1,d_2 \geq 3$, we have $R_{d_1,d_2}(W_{d_1}\otimes W_{d_2})\le 2d_1+2d_2-1$. 
\end{proposition}

\begin{proof}
Let $L_1=\{o_1\}\times \mathbb{P}^1$, $L_2=\mathbb{P}^1\times \{o_1\}$ and $Z = Z_2$ as in Section \ref{section: Preliminaries} (Notation). Let $D\in |\Oo _{\PP^1\times \PP^1}(1,1)|$ be any smooth curve passing through $o$. Let $C:= L_1 L_2 D \in \vert \calO_{\bbP^1 \times \bbP^1} (2,2) \vert$.

\begin{claim*}
 $Z\subset C$. 
\end{claim*}

\begin{quote}
\begin{proof}[Proof of Claim]
Observe first that the double point $2o$ is contained in the intersection $Z \cap (L_1 \cup L_2)$. Indeed $L_1 \cup L_2$ is singular at $o$, therefore $2o \subseteq L_1 \cup L_2$. Moreover $2o \subseteq Z$: using local coordinates $\eta_1,\eta_2$ in a neighborhood of $o$, we have $I_Z = (\eta_1^2 , \eta_2^2)$ and $I_{2o} = (\eta_1,\eta_2)^2 = (\eta_1^2 , \eta_1\eta_2, \eta_2^2)$, so that $I_Z \subseteq I_{2o}$ therefore $2o \subseteq Z$.  Since $\deg(2o) = 3$, we have $\deg(Z\cap (L_1\cup L_2))\geq 3$ and by Eqn. \eqref{eqn: degree as residue plus intersection}, we obtain $\deg (\Res _{L_1 L_2}(Z)) \leq 1$ because $\deg(Z) = 4$. This shows $\Res _{L_1 L_2}(Z)) \subseteq \{ o \}$ and since $o \in D$, we conclude.
\end{proof}\end{quote}

The curve $C$ is reduced and connected: it is reduced, because $L_1$ and $L_2$ are distinct and they are not contained in $D$ because $D$ is smooth; it is connected because the three components are connected and they intersect at $o$.  Consider the exact sequence
\[
0\to \Oo _{\PP^1\times \PP^1}(d_1-2,d_2-2)\xto{\cdot C} \Oo _{\PP^1\times \PP^1}(d_1,d_2)\to \Oo _{C}(d_1,d_2) \to 0. 
\]
By K\"unneth's formula (see e.g. p. 58 in \cite{GrifHar:PrinciplesAlgebraicGeometry}) we have $h^1(\Oo _{\PP^1\times \PP^1}(d_1-2,d_2-2)) = h^1(\Oo _{\PP^1}(d_1-2))h^0(\Oo _{\PP^1}(d_2-2)) + h^0(\Oo _{\PP^1}(d_1-2)) h^1(\Oo _{\PP^1}(d_2-2)) = 0$, because $h^1(\Oo _{\PP^1}(d)) = 0$ if $d \geq 0$ and by hypothesis $d_1,d_2 \geq 3$. So the long exact sequence in cohomology of the sequence above provides $h^0(\Oo _C(d_1,d_2)) = h^0(\Oo _{\PP^1\times \PP^1}(d_1,d_2))- h^0(\Oo _{\PP^1\times \PP^1}(d_1-2,d_2-2)) = (d_1+1)(d_2+1) -(d_1-1)(d_2-1) = 2d_1+2d_2$. The linear span $E:= \langle \nu_{d_1,d_2}(C)\rangle$ has dimension at most $2d_1+2d_2-1$. This argument proves the upper bound $R_{d_1,d_2}(W_{d_1}\otimes W_{d_2})\le 2d_1+2d_2$.

The rest of the argument will provide the additional increment by $1$. Since $Z\subset C$ and $W_{d_1}\otimes W_{d_2}\in \langle \nu _{d_1,d_2}(Z)\rangle$, we compute $R_{\nu_{d_1,d_2}(C)}(W_{d_1}\otimes W_{d_2})$; since $\nu_{d_1,d_2}(C)\subset \nu_{d_1,d_2}(\mathbb{P}^1\times \mathbb{P}^1)$, we obtain the upper bound $R_{d_1,d_2}(W_{d_1}\otimes W_{d_2}) \leq R_{\nu_{d_1,d_2}(C)}(W_{d_1}\otimes W_{d_2})$. Let $H$ be a generic hyperplane in $E$ through $W_{d_1} \otimes W_{d_2}$: since $\nu_{d_1,d_2}(C)$ is a curve, by Bezout's Theorem on $\bbP^1 \times \bbP^1$ (see e.g. \cite{EisHar:3264}, Ch. 1) we have that $H \cap \nu_{d_1,d_2}(C)$ is a collection of $\deg(\nu_{d_1,d_2}(C)) = 2d_1 + 2d_2$ points. We want to show this choice of $H$ is generic enough so that $\langle H \cap  \nu_{d_1,d_2}(C) \rangle = H$. Observe that $h^1 (\Ii _{\nu_{d_1,d_2}(C),E} ) = 0$: this follows from the restriction exact sequence 
\[
 0 \to \calI_{\nu_{d_1,d_2}(C),E} \to \calO_E \to \calO_{\nu_{d_1,d_2}(C)} \to 0
\]
since $h^0(\calO_{\nu_{d_1,d_2}(C)}) = 1$ (because $C$ is reduced and connected) and the restriction map $H^0(\calO_E) \to H^0(\calO_{\nu_{d_1,d_2}(C)})$ is surjective. Now consider the exact sequence in $E$:
\begin{equation}\label{eqn: exact sequence in Prop Rd1d2}
 0\to \Ii _{\nu_{d_1,d_2}(C),E} \xto{\cdot H} \Ii _{\nu_{d_1,d_2}(C) ,E}(1)\to \Ii _{H \cap \nu_{d_1,d_2}(C),H}(1)\to 0.
\end{equation}
Since $E$ is defined as the span of $\nu_{d_1,d_2}(C)$, we have that $\nu_{d_1,d_2}(C)$ does not have linear equations in $E$, namely $h^0 (\Ii _{\nu_{d_1,d_2}(C) ,E}(1)) = 0$. The long exact sequence in cohomology of \eqref{eqn: exact sequence in Prop Rd1d2} provides that $h^0 (\Ii _{H \cap \nu_{d_1,d_2}(C),H}(1)) = 0$, which means that $H \cap \nu_{d_1,d_2}(C)$ has no linear equations in $H$, therefore $\langle H \cap  \nu_{d_1,d_2}(C) \rangle = H$. Now, since $\dim H = 2d_1 + 2d_2 - 2$, the points of $H \cap \nu_{d_1,d_2}(C) $ are linearly dependent. This proves that every point of $E$ has $\nu_{d_1,d_2}(C)$-rank at most $2d_1 + 2d_2 -1$, and this concludes the proof.
\end{proof}

Fix integers $k\ge 2$ and $d_i\ge 3$, $1\le i \le k$ and let $T = W_{d_1} \ootimes W_{d_k}$. We generalize the construction of Proposition \ref{prop: upper bound Wd1 ot Wd2} to give an upper bound for $R_{d_1 \vvirg d_k}(T)$, which improves the bound in \cite[Cor. 11]{cjz} by roughly a factor of $2$. The argument is similar to the first part of Proposition \ref{prop: upper bound Wd1 ot Wd2}, where the bound $2d_1+2d_2$ is provided. However, the more general setting makes it hard to use a geometric argument to prove the analog of the inclusion $Z \subseteq C$ of the Claim of Proposition \ref{prop: upper bound Wd1 ot Wd2}: in the proof of the next result we use a Gr\"obner degeneration argument, which exploits the combinatorics of the ideals involved, making it possible to prove the desired inclusion in general.

\begin{theorem}\label{bo1}
 For every $k \geq 2$, every $d_1 \vvirg d_k \geq 3$, we have
 \[
 R_{d_1, \ldots , d_k}(W_{d_1}\otimes \cdots \otimes W_{d_k}) \leq 2^{k-1}(d_1 + \cdots + d_k).
\]
\end{theorem}

\begin{proof}
For every $\Lambda \subseteq \{ 1 \vvirg k\}$, let $\delta_\Lambda: \bbP^1 \to \bbP^1 \ttimes \bbP^1$ be the embedding defined by $\delta_\Lambda(z) = (\zeta_1 \vvirg \zeta_k)$ where $\zeta_i = z$ if $i \in \Lambda$ and $\zeta_i = o_1\in \mathbb{P}^1$ if $i \notin \Lambda$. For every $\Lambda$, let $B_\Lambda$ be the image of $\delta_\Lambda$: $B_\Lambda$ embeds in $\bbP(S^{d_1} \bbC^2 \ootimes S^{d_k} \bbC^2)$ via $\nu_{d_1 \vvirg d_k}$ and its image is a rational normal curve of degree $\sum_{i \in \Lambda} d_i$. In particular, for every $\Lambda$, $\dim (\langle \nu_{d_1 \vvirg d_k} (B_\Lambda) \rangle) \leq \sum_{i \in \Lambda} d_i$.

Let $C = \bigcup_{\Lambda \subseteq \{ 1 \vvirg k\}} B_\Lambda$. Notice that $B_{\Lambda_1} \neq B_{\Lambda_2}$ if $\Lambda_1 \neq \Lambda_2$, so $C$ is a reduced curve having $2^{k}-1$ irreducible components (note that $B_\emptyset = \{o\}$ is not a curve); moreover since $o \in B_\Lambda$ for every $\Lambda$, we have that $C$ is connected.  Our goal is to prove that $Z \subseteq C$. To do this, we work locally.

We consider local coordinates on $\bbP^1 \ttimes \bbP^1$ as follows: let $\xi_i,\eta_i$ be the basis dual to $x,y$ on the $i$-th copy of $\bbP^1$; after the dehomogenization $\xi_1 = \cdots = \xi_k = 1$ we consider $\eta_1 \vvirg \eta_k$ local coordinates on $\bbA^k = \bbC^k = (\bbP^1 \ttimes \bbP^1) \setminus \{( y \vvirg y) \}$. Let $R = \bbC[\bbA^k] = \bbC[\eta_1 \vvirg \eta_k]$ and write $R_t$ for the homogeneous component of degree $t$ and $R_{\leq t} = R/R_{t+1}$, the latter being a zero-dimensional ring, which is also a finite dimensional vector space. In these coordinates $Z$ is the scheme cut out by the ideal $I_Z = (\eta_1^2 \vvirg \eta_k^2)$ and the point $o$ is cut out by the ideal $I_{Z_{\red}} = \sqrt{I_Z} = (\eta_1 \vvirg \eta_k)$.  The coordinate ring $\bbC[Z] = R/ I_Z$ has a basis given by (the images in the quotient of) the square-free monomials of $R$; if $\Lambda \subseteq \{ 1 \vvirg k \}$, we denote $\eta^\Lambda = \prod_{i \in \Lambda} \eta_i$. 

Locally, the map $\delta_\Lambda $ embeds the affine line $\bbA^1$ (with a coordinate $\eta$) to the diagonal of the coordinate plane defined by $\Lambda$, namely $\delta(\eta) = (\zeta_1 \vvirg \zeta_k)$ where $\zeta_i = \eta$ if $i \in \Lambda$ and $\zeta_i = 0$ if $i \notin \Lambda$. Denote by $B_\Lambda^{\circ}$ the image of $\delta_\Lambda$ in $\bbA^k$, so that $I_{B_\Lambda^{\circ}} = (\eta_i : i \notin \Lambda)$; let $C^\circ = \bigcup_{\Lambda \subseteq \{ 1 \vvirg k\}} B_\Lambda ^\circ$.

\begin{claim*}
 We have $I_{C^\circ} = ( \eta_i \eta_j (\eta_i - \eta_j) : i,j = 1 \vvirg k)$.
\end{claim*}
\begin{quote}
\begin{proof}
Let $J = ( \eta_i \eta_j (\eta_i - \eta_j) : i,j = 1 \vvirg k) \subseteq R$. First we prove $\sqrt{J} = \sqrt{I_{C^\circ}}$, namely that $J$ and $I_{C^\circ}$ define the same variety set theoretically. Let $p \in C^\circ$: then we have $p \in B_\Lambda$ for some $\Lambda$, so that $\eta_i(p) = 0$ if $i \neq \Lambda$ and $\eta_i(p) = \eta_j(p)$ if $i,j \in \Lambda$; in particular, the generators of $J$ vanish at $p$. Conversely if $p$ is a point in the (support of the) scheme defined by $J$, then for every two indices $i,j$, we have $\eta_i(p) \eta_j(p) (\eta_i(p) - \eta_j(p)) = 0$, so either $\eta_i(p) = 0$, or $\eta_j(p) = 0$ or $\eta_i(p) = \eta_j(p)$; this shows $p \in C^\circ$. Now, since $C^\circ$ is reduced, we have $\sqrt{I_{C^\circ}} = I_{C^\circ}$, so $\sqrt{J} = I_{C^\circ}$ and therefore $J \subseteq I_{C^\circ}$. 

Since both $I_{C^\circ}$ and $J$ are homogeneous, we regard them as ideals of schemes in $\bbP^{k-1} = \bbP \bbC^k$. In particular, the projectivization of $C^\circ$ is the set of points $\{ (\zeta_1 \vvirg \zeta_k) \in \bbP^{k-1}: \zeta_i \in \{ 0,1 \}\}$ which is a zero-dimensional (smooth) scheme of degree $\deg(C^\circ) = 2^k -1$. So the Hilbert polynomial of $I_{C^\circ}$ is $\HP_{I_{C^\circ}} = 2^k-1$. Since $J \subseteq I_{C^\circ}$, we have $\HP_J \geq \HP_{I_{C^\circ}}$. We will prove that the opposite inequality holds as well, providing equality of Hilbert polynomials, and therefore of the ideals.

Consider the lexicographic monomial order on $\bbC[\eta_1 \vvirg \eta_k]$ (ordered according to the indices): denote by $\LT(J)$ the monomial ideal of leading terms of $J$ and by $U$ the monomial ideal generated by the leading terms of the binomials $\eta_i \eta_j (\eta_i - \eta_j) $ for every $i,j$. In particular $U = ( x_i^2 x_j : 1 \leq i< j \leq k )$ and $U \subseteq \LT(J)$. By \cite{CLO}, Prop. 4 in Ch.9, \S3, we have the equality of Hilbert polynomials $\HP_{\LT(J)} = \HP_{J}$ and since $U \subseteq \LT(J)$, we have $\HP_U \geq \HP_{\LT(J)}$. We will show that $U$ has constant Hilbert polynomial $\HP_U = 2^k-1$; this will conclude the proof.

Fix $t \gg k$. We prove $\HP_U(t) = 2^k-1$. Let $\alpha$ be a multi-index with $\vert \alpha \vert = t$ and let $\eta^\alpha$ be the corresponding monomials. We have $\eta^\alpha \in U_t$ if and only if there are $i,j$ with $i<j$ such that $\alpha_i \geq 2$ and $\alpha_j \geq 1$. In particular, $\HP_U(t) = \dim \bbC[\eta_1 \vvirg \eta_k]_t / U_t$ is equal to the cardinality of the set $\calA_t = \bigcup_{\ell = 1 \vvirg k} \calA_{t,\ell}$, where
\[
  \calA_{t,\ell} = \{ \alpha \in \bbN^k : \alpha_i = 0,1 \text{ if $i< \ell$}, \ \alpha_\ell \geq 2 \text{ and } \alpha_i = 0 \text{ if $i > \ell$} \}.
\]
It is clear that the $\calA_{t,\ell}$'s are disjoint. The elements of $\calA_{t,\ell}$ are in one-to-one correspondence with the subsets of $\{ 1 \vvirg \ell-1\}$, so $\vert \calA_{t,\ell} \vert = 2^{\ell-1}$. We conclude $\vert \calA_t \vert = \sum_{\ell = 1}^k  \vert \calA_{t,\ell} \vert = \sum_{\ell = 1}^k 2^{\ell -1} = 2^k -1$. This shows $\HP_U = 2^k-1$ and we conclude.
\end{proof}
\end{quote}

From the Claim above, notice that $I_Z \supseteq I_{C^\circ}$, so we obtain $Z \subseteq C^\circ \subseteq C$. We obtain $T \in \langle \nu_{d_1 \vvirg d_k}(C) \rangle$ and therefore $R_{d_1 \vvirg d_k}(T) \leq R_{\nu_{d_1\vvirg d_k}(C)}(T)$. Moreover $\langle  \nu_{d_1 \vvirg d_k}(C) \rangle$ \- $  = $ $\sum_{\Lambda \in \{ 1 \vvirg k\} } \langle  \nu_{d_1 \vvirg d_k}(B_\Lambda) \rangle$. In particular, for every $\Lambda$, there exists $p_\Lambda \in \langle \nu_{d_1 \vvirg d_k}(B_\Lambda) \rangle$ such that $T \in \langle p_\Lambda : \Lambda \subseteq \{ 1 \vvirg k\} \rangle$.

For every $\Lambda$, we have $R_{\nu_{d_1 \vvirg d_k}(C)} (p_\Lambda) \leq R_{\nu_{d_1 \vvirg d_k}(B_\Lambda)} (p_\Lambda)\leq \sum_{i \in \Lambda} d_i$, because $\nu_{d_1 \vvirg d_k}(B_\Lambda)$ is a rational curve of degree $\sum_{i \in \Lambda} d_i$. We conclude 
\[
 R_{\nu_{d_1 \vvirg d_k}(C)} (T) \leq \sum_{\Lambda \subseteq \{ 1 \vvirg k\}} R_{\nu_{d_1 \vvirg d_k}(C)} (p_\Lambda) \leq \sum_{\Lambda \subseteq \{ 1 \vvirg k\}} \left( \textsum_{i \in \Lambda} d_i \right) = 2^{k-1}(d_1 + \cdots + d_k).
\]
\end{proof}

In the case $d_i=d$ for all $i$, we obtain the following result.

\begin{corollary}\label{bo3}
It $k \geq 2$ and $d \geq3$, then $R_{d \vvirg d}(W_d^{\otimes k})\leq 2^{k-1} k d$. 
\end{corollary}

We point out that part of the proof of Theorem \ref{bo1} does not require the underlying field to be $\bbC$. In particular we have

\begin{remark}\label{bo2}
The first part of the proof of Theorem \ref{bo1} and the Claim is valid on every field $\bbF$. The last part of the proof requires the underlying field to be $\bbC$ because it uses that if $B$ is a rational normal curve of degree $d$ then every element of $\langle B \rangle$ satisfies $R_B(p) \leq d$. However, we can use a slightly different argument to obtain almost the same upper bound if $\vert \bbF \vert \geq d_1 + \cdots + d_k + 1$. Indeed, for every $\Lambda \subseteq \{ 1 \vvirg k\}$, let $d_\Lambda = \sum_{i \in \Lambda}d_i$ and let $\Omega_\Lambda = \{ q^{(\Lambda)}_0 \vvirg q^{(\Lambda)}_{d_\Lambda}\}$ be $d_\Lambda+1$ distinct points of $\nu_{d_1 \vvirg d_k}(B_\Lambda)$ with $q_0^{(\Lambda)} = o$. Then $W_{d_1} \ootimes W_{d_k} \in \left\langle \bigcup_{\Lambda} \Omega_\Lambda \right\rangle$ and the union contains at most $1 + 2^{k-1} (d_1 + \cdots + d_k)$ points because $o \in \Omega_\Lambda$ for every $\Lambda$. We conclude $R_{d_1 \vvirg d_k}( W_{d_1} \ootimes W_{d_k}) \leq 1 + 2^{k-1}(d_1 + \cdots + d_k)$ over every field $\mathbb{F}$ with $\vert \bbF \vert \geq d_1 + \cdots + d_k + 1$.
\end{remark}

\subsection{Other tensors}\label{subsec: upper bounds others}

The next result deals with tensors of low cactus rank spanned by zero-dimensional schemes supported at $o_1$. For every $b$, denote by $Z[b]$ the zero-dimensional scheme of degree $b$ in $\bbP^1$, supported at $o_1$. In particular, in the coordinates $\partial_x,\partial_y$ dual to $x,y$, we have $I_{Z[b]} = (\partial_y^b)$; moreover for every $d \geq b$, we have $h^0( \calI_{Z[b]}(d)) = d-b+1$ and $h^1 (\calI_{Z[b]}(d)) = 0$. 

Fix $d$ and $b$ with $b \leq \lfloor \frac{d}{2} \rfloor$ and let $T_{d,b}$ be a tensor in $S^d \bbC^2$ with cactus rank $c_d(T_{d,b}) = b$ and $T_{d,b} \in \langle \nu_{d} (Z[b]) \rangle$. For instance $T_{d,2} = W_d$ (from Lemma \ref{lemma: cactus of Wd1 tatata Wdk}) and $T_{d,3} = x^{d-2} y^2$ (see e.g. Remark \ref{remark: sylvester thm}). By Sylvester Theorem, we have $R_{d}(T_{d,b}) = d+2 -b$.

The argument followed in the proof of the next result is similar to the one of Proposition \ref{prop: upper bound Wd1 ot Wd2} and Theorem \ref{bo1}: we determine a curve in $\bbP^1 \times \bbP^1$ containing the zero-dimensional scheme spanning the tensor of interest, and we give a bound on the rank providing a spanning set of the linear span of such curve.

\begin{proposition}\label{b1}
Fix integers $d_1,d_2,b_1,b_2$ with $2 \leq b_i \leq \frac{d_i}{2}$. Let $T_i = T_{d_i,b_i}$ and $T = T_1 \otimes T_2$. Then $R_{d_1,d_2}(T)\leq  d_1 b_2 + d_1 b_2 + b_1b_2-1$.
\end{proposition}
\begin{proof}
Let $B_i = Z[b_i]$ be a zero-dimensional scheme evincing $c_{d_i}(T_i)$. Let $B = B_1 \times B_2 \in \bbP^1 \times \bbP^1$.  By K\"unneth's formula, we have $h^0(\calI_{B}(b_1,b_2)) = b_1 + b_2 + 1$. Let $C$ be a generic element in $\vert \calI_{B}(b_1,b_2) \vert$. 

\begin{claim*} $C$ is reduced.
\end{claim*}
\begin{quote}
\begin{proof}[Proof of the Claim]
Since $C$ is effective and Cartier, if $C$ is not reduced, then it has at least one irreducible component $D'$ appearing with multiplicity $e\ge 2$. Since $\bbP^1 \times \bbP^1$ is smooth, by Bertini's Theorem (see e.g. \cite[Corollary III.10.9 and Remark III.10.9.2]{hart}), $C$ is smooth away from the base locus of $|\Ii _B(b_1,b_2)|$; since $C$ is not smooth along $D'^e$, we deduce that $D'^e$ is contained in the base locus. Now, $L_1^{b_1} L_2^{b_2}\in |\Ii _B(b_1,b_2)|$; by definition of base locus, we have that $L_1^{b_1}L_2^{b_2}$ contains the base locus, and in particular $D'^e \subseteq  L_1^{b_1} L_2^{b_2}$. This shows $D'^e = L_1^{c_1} L_2^{c_2}$ for some $c_1,c_2 \geq 0$ with $c_1 + c_2 = e$.

Since $D'$ is irreducible, we deduce that either $D' =L_1$ and $e = c_1 \le b_1$ or $D'= L_2$ and $e = c_2 \leq b_2$. Without loss of generality, assume $D' = L_1$. We obtain that $L_1^e$ is contained in the base locus of $|\Ii _B(b_1,b_2)|$, namely $|\Ii _B(b_1,b_2)| = L_1^e \cdot \vert \Ii _{\Res_{L_1^e}(B)} (b_1-e,b_2) \vert$. Notice that $\Res_{L_1^e}(B) = \Res_{L_1^e}(Z[b_1] \times Z[b_2]) = Z[b_1-e] \times Z[b_2]$. Again by K\"unneth's formula, we have $h^0 (\Ii _{\Res_{L_1^e}(B)} (b_1-e,b_2)) = b_1 -e + b_2 + 1$; but since $h^0 (\Ii _B(b_1,b_2)) = h^0( \Ii _{\Res_{L_1^e}(B)} (b_1-e,b_2))$, we obtain $b_1 - e + b_2 + 1 = b_1  + b_2 + 1$, providing $e = 0$ which is a contradiction. 
\end{proof}\end{quote}

As in the proof of Theorem \ref{bo1}, we have that $R_{d_1 , d_2}(T) \leq R_{\nu_{d_1,d_2} ( D )}(T) \leq \dim \langle \nu_{d_1,d_2} ( D )\rangle$. Using $h^1(\Oo _{\bbP^1 \times \bbP^1}(d_1-b_1,d_2-b_2)) =0$, the restriction exact sequence of $D$ in $\bbP^1 \times \bbP^1$ provides $1 + \dim \langle \nu_{d_1,d_2} ( D )\rangle = h^0(\Oo _D(d_1,d_2)) = h^0(\Oo _{\bbP^1 \times \bbP^1}(d_1,d_2)) -h^0(\Oo _{\bbP^1 \times \bbP^1}(d_1- b_1,d_2-b_1)) = (d_1+1)(d_2+1)-(d_1-b_1+1)(d_2-b_2+1) = d_1 b_2 + d_2 b_1 + b_1b_2$.
Hence, $R_{d_1 ,d_2}(T)\le d_1 b_2 + d_1 b_2 + b_1b_2-1$. 
\end{proof}

\section{Lower bounds on the partially symmetric rank}\label{section: lower bounds}

In this section, we provide lower bounds on partially symmetric rank of products of $W$-states.

In the following lemma, we show that if $A \subseteq (\bbP^1)^k$ is a zero-dimensional scheme such that $W_{d_1} \ootimes W_{d_k} \in \langle \nu_{d_1 \vvirg d_k}(A) \rangle$, then $A$ and $Z$ have the same equations up to multidegree $(d_1-1 \vvirg d_k-1)$. This can be interpreted as a minimality result of $Z$ among zero-dimensional schemes evincing the cactus rank of products of $W$-states.
\begin{lemma}\label{lemma: I222 of Z is maximal among schemes spanning T}
Let $k$ be a positive integer and let $d_1 \vvirg d_k \geq 3$ be nonnegative. Let $T:= W_{d_1} \ootimes W_{d_k}$. Let $A \subseteq \bbP^1 \ttimes \bbP^1$ be a zero-dimensional scheme such that $T \in \langle \nu_{d_1 \vvirg d_k}(A) \rangle$. Then $\vert \calI_A(a_1 \vvirg a_k ) \vert \subseteq \vert \calI_Z(a_1 \vvirg a_k ) \vert $ for every choice of $a_i$ with $0 \leq a_i \leq d_i-1$. 
\end{lemma}
\begin{proof}
If $A' \subseteq A$ then $\vert \calI_A(a_1 \vvirg a_k) \vert \subseteq \vert \calI_{A'}(a_1 \vvirg a_k)\vert$, so, without loss of generality, we may assume that $A$ is minimal, in the sense that there is no $A' \subsetneq A$ with $T \in \langle \nu_{3 \vvirg 3}(A') \rangle$.

First assume $a_i = d_i-1$ for every $i$. If $\vert \calI_A (d_1-1 \vvirg d_k-1) \vert = \emptyset$, then there is nothing to prove. Let $D \in \vert \calI_A (d_1-1 \vvirg d_k-1) \vert$; in particular we have $\Res_D(A \cup Z) = \Res_D(Z)$.

We have $h^1(\calI_Z(1 \vvirg 1)) = 0$, therefore $h^1(\calI_{\Res_D(Z)}(1 \vvirg 1) ) = 0$ as well. Since $\calI_D \simeq \calO_{(\bbP^1)^{\times k}}(1-d_1 \vvirg 1-d_k)$, we deduce $\calI_{\Res_D(Z)}(1 \vvirg 1) \simeq  \calI_{\Res_D(Z)} (d_1 \vvirg d_k) \otimes \calI_D$. Moreover $\calI_{\Res_D(Z)} (d_1 \vvirg d_k) \simeq \calI_{\nu_{d_1 \vvirg d_k}(\Res_D(Z))}(1)$. Since we have the equality $\nu_{d_1 \vvirg d_k} (\Res_D (Z)) = \Res_{\nu_{d_1 \vvirg d_k}(D)}(\nu_{d_1 \vvirg d_k}(Z))$, we apply Lemma \ref{lemma: 5.1 plus} with $\nu_{d_1 \vvirg d_k}(Z)$ playing the role of $B$ and $\nu_{d_1 \vvirg d_k}(D)$ playing the role of $D$. We obtain $\nu_{d_1 \vvirg d_k}(Z) \subseteq \nu_{d_1 \vvirg d_k}(D)$, and equivalently $Z \subseteq D$, providing $D \in \vert \calI_Z( d_1 -1 \vvirg d_k-1)\vert$.

Now, suppose $a_i \leq d_i-1$ arbitrary. Let $D \in \vert \calI_A (a_1 \vvirg a_k) \vert$. Since $\vert \calO_{\bbP^1 \ttimes \bbP^1}(d_1-1-a_1 \vvirg d_k-1-a_k) \vert $ has no base points and $A \cup Z$ has finite support, there exists $G \in \vert \calO_{\bbP^1 \ttimes \bbP^1}(d_1-1 - a_1 \vvirg d_k-1-a_k) \vert$ such that $G \cap (A \cup Z) = \emptyset$. Now, $DG \in \vert\calI_A(d_1-1 \vvirg d_k-1) \vert$ and we conclude using the first part of the proof.
\end{proof}

The following result is a weaker version of Theorem 10 in \cite{cf}: it gives the same lower bound as \cite{cf} for the tensor product $W_3 \otimes W_3$, but only restricting to the partially symmetric case. Our proof uses completely different techniques: we essentially perform a case by case analysis on the different possible bi-degrees of a divisor on $\bbP^1 \times \bbP^1$; some of the arguments that we use are completely general and may be found useful to address other problems in the partially symmetric setting.

\begin{proposition}\label{fingercrossed}
We have $R_{3,3}(W_3\otimes W_3) = 8$.
\end{proposition}

\begin{proof}

By \cite{cjz}, we have $R_{3,3}(T) \leq 8$.

In \cite{YCGD10}, the bound $R_{1,1,1}( W_3 \otimes W_3) \geq 7$ is given, hence we have $R_{3,3} (T) \geq 7$. Suppose by contradiction $R_{3,3} (T) = 7$ and let $S$ be the set of 7 distinct points in $\PP^1\times \PP^1$ computing the rank of $T$. Let $E = Z \cup S$. From Lemma \ref{lemma: I222 of Z is maximal among schemes spanning T}, we have $\vert \calI_S(2,2) \vert \subseteq \vert \calI_Z(2,2) \vert$, and therefore $\vert \calI_S(2,2) \vert = \vert \calI_E(2,2) \vert$ since $E = Z \cup S$.

Observe that $\deg(E) \geq 10$. Indeed, $Z$ is supported at the single point $o$, and either $o \in S$ or $o \notin S$; in the first case we have $\deg(E) = 10$ and in the second case we have $\deg(E) = 11$.

By Bezout's Theorem in $\bbP^1 \times \bbP^1$, two elements in $\vert \calO_{\bbP^1 \times \bbP^1}(2,2) \vert$ either intersect in a zero-dimensional scheme of degree $2 \cdot 2 + 2\cdot 2  = 8$ or their intersection contains a divisor. The intersection of any two elements of $\vert \calI_S(2,2) \vert$ contains $E$ and $\deg(E) \geq 10$, so $\vert \calI_S(2,2)\vert$ has positive dimensional base locus. Let $B$ be the union of the positive dimensional components of the base locus of $\vert \calI_S(2,2) \vert$; in particular $B$ is a divisor on $\bbP^1 \times \bbP^1$, therefore $B \in \vert \calO_{\bbP^1 \times \bbP^1}(a,b) \vert$ for some $a,b$ such that $ a,b \leq 2$. In particular, we have $\vert \calI_B(2,2) \vert = B \cdot \vert \calO_{\bbP^1 \times \bbP^1} (2-a,2-b)\vert$ and $\vert \calI_E( 2,2) \vert = B \cdot \vert \calI_{\Res_B(E)}(2-a,2-b) \vert$. By definition of $B$, $\vert \calI_{\Res_B(E)}(2-a,2-b)\vert $ does not contain any divisor in its base locus. Again by Bezout's Theorem in $\bbP^1 \times \bbP^1$, two generic elements in $\vert \calO_{\bbP^1 \times \bbP^1}(2-a,2-b) \vert$ intersect in a zero-dimensional scheme of degree $ 2(2-a)(2-b)$, therefore $\deg (\Res_B(E)) \leq 2(2-a)(2-b)$.

For every possible pair $(a,b)$, we find a contradiction and this will conclude the proof. Let $L_1,L_2$ as in Section \ref{section: Preliminaries} (Notation).

\begin{enumerate}
\item $(a,b) = (2,2)$. In this case, we have $\{ B\} = \vert \calI_B(2,2) \vert = \vert \calI_E(2,2) \vert = \vert \calI_S(2,2) \vert$. On the other hand $\dim \vert \calO_{\bbP^1 \times \bbP^1}(2,2) \vert = 8$ and $\deg(S) = 7$, so $\dim \vert \calI_S(2,2) \vert \geq 1$. This gives a contradiction to $\{ B\} = \vert \calI_B(2,2) \vert$ as $\dim\{ B \} = 0$.

 \item $(a,b) \in \{ (2,1) ,(1,2) \}$. Up to exchanging the roles of the two factors, assume $(a,b) = (2,1)$. We obtain $\deg (\Res_B(E) ) = 0$, so $E \subseteq B$. Since the double point $2o$ is contained in $Z \subseteq E \subseteq B$, we deduce that $B$ is singular at $o$. Elements of $\vert \calO_{\bbP^1 \times \bbP^1} (2,1) \vert$ that are singular at $o$ are of one of the following forms: either $L_1^2  L_2'$ for some $L_2' \in \vert \calO_{\bbP^1 \times \bbP^1}(0,1) \vert$ or $L_1L_1'L_2$ for some $L' \in \vert \calO_{\bbP^1 \times \bbP^1} (1,0)\vert$. However, since $Z \subseteq B$, the second case can only occur if $L_1' = L_1$ because $Z$ is not contained in $L_1 L_2$; therefore in either case, we have $B = L_1^2 L_2'$ for some $L_2' \in \vert \calO_{\bbP^1 \times \bbP^1}(0,1) \vert$ (possibly $L_2' = L_2$). Since $S$ is reduced, we have $S \subseteq B_{\red} = L_1 L_2'$. Since $\deg(L_1 \cap Z) = 2$, we have $\deg( \Res_{L_1 L_2'} (Z) ) \leq 2$, and therefore $\deg( \Res_{L_1  L_2'} (E)) \leq 2$, that provides $h^1(\Ii _{\Res _{L_1 L_2'}(E)}(2,2)) =0$ (see e.g. \cite[Prop. 7.3]{hart}). By Lemma \ref{lemma: 5.1 plus}, we conclude $Z \subseteq L_1 L_2'$, which is a contradiction because the zero-dimensional scheme supported at $o$ contained in $L_1 L_2'$ is either $\{o \}$ itself (if $L_2 \neq L_2'$) or $2o$ (if $L_2 = L_2'$).
 
 \item $(a,b) = (1,1)$. We obtain $\deg(\Res_B(E)) = 2$ and therefore $h^1(\Ii _{\Res_B(E)}(2,2))=0$ (again by \cite[Prop. 7.3]{hart}). By Lemma \ref{lemma: 5.1 plus} we have $Z\subset B$, which is false, because $h^0(\Ii _Z(1,1)) =0$.

 \item $(a,b) \in \{(2,0),(0,2)\}$. Up to exchanging the roles of the two factors, assume $(a,b) = (2,0)$. We obtain $\deg(\Res_B(E)) = 0$, namely $E \subseteq B$ and therefore $Z \subseteq B$. This implies $B = L_1^2$, and since $S$ is reduced, we obtain $S \subseteq B_{\red} = L_1$, that gives $\langle S \rangle = L_1$, and in particular $T \in \langle \nu_{3,3}(L_1)\rangle$, which is false because elements in $\langle \nu_{3,3}(L_1)\rangle$ are of the form $[F \otimes x^3]$ for some $F \in S^3 \bbC^2$. 
 
 \item $(a,b) \in \{ (1,0),(0,1)\}$. Up to exchanging the roles of the two factors, assume $(a,b) = (1,0)$. We obtain $\deg(\Res_B(E)) = 4$. Since $10 \leq \deg (E) \leq 11$, we obtain $6 \leq \deg(\Res_B(E)) \leq 7$. By minimality of $S$, we have $\deg(S \cap B) \leq 4$ because any subset $S'$ in $B$ has $h^1(\calI_{S'}(3,3)) > 0$, whereas $h^1(\calI_S(3,3)) = 0$. This shows $B \cap Z \neq \emptyset$, and therefore $B = L_1$. We have $\deg(Z \cap L_1) = 2$, so $\Res_{L_1} (Z) = 2$. On the other hand, since $\deg(L_1 \cap S) \leq 4$, we have $\deg( \Res_{L_1}(S)) \geq 3$ and $o \notin  \Res_{L_1}(S)$ because $o \in L_1$. This shows $\Res_{L_1} (E) = \Res_{L_1}(S) \sqcup \Res_{L_1}(Z)$ and passing to the degrees $\deg(\Res_{L_1} (E)) = \deg(\Res_{L_1}(S)) + \deg(\Res_{L_1}(Z)) \geq 3 + 2 = 5$, whereas $\deg(\Res_B(E)) = 4$. This gives a contradiction.
 \end{enumerate}
 \end{proof}

The lower bound in the following proposition applies to tensor rank.
 
 \begin{proposition}\label{lower1}
Fix nonnegative integers $d_i$, $1\le i \le k$. Then 
\[
R_{1 \vvirg 1}(W_{d_1}\otimes \cdots \otimes W_{d_k}) \geq d_1+\cdots +d_k-k+1 .  
\]
\end{proposition}
\begin{proof}
Fix $k = 2$ and consider $W_{d_1} \otimes W_{d_2} \in S^{d_1} \bbC^2 \otimes S^{d_2} \bbC^2 \subseteq (\bbC^2)^{\otimes d_1} \otimes (\bbC^2)^{\otimes d_2}$. Define a linear map $\phi: \bbC^2 \otimes \bbC^2 \to \bbC^2 $ via $\phi(\ell_1 \otimes \ell_2) = (\partial_{x_1} \ell_1 \partial_{y_2} \ell_2 + \partial_{y_1} \ell_1 \partial_{x_2} \ell_2 )  x + (\partial_{y_1} \ell_1 \partial_{y_2} \ell_2) y$ where we consider basis $x_i,y_i$ on the two copies of $\bbC^2$ defining the domain and $x,y$ for the codomain.

Let $\Phi : (\bbC^2)^{\otimes d_1} \otimes (\bbC^2)^{\otimes d_2} \to (\bbC^2)^{\otimes d_1-1} \otimes \bbC \otimes (\bbC^2)^{\otimes d_2-1}$ be defined by $\Phi = \id_{\bbC^2}^{\otimes d_1-1} \otimes \phi \otimes  \id_{\bbC^2}^{\otimes d_2-1}$, namely just performing $\phi$ on the tensor product of the last factor of $(\bbC^2)^{\otimes d_1}$ and the first factor of $(\bbC^2)^{\otimes d_2}$. It is immediate that the Segre variety $\nu_{1 \vvirg 1} ( (\bbP^1 )^{d_1 + d_2}) \subseteq (\bbC^2)^{\otimes d_1 + d_2}$ is mapped via $\Phi$ to the Segre variety $\nu_{1 \vvirg 1} ( (\bbP^1 )^{d_1 + d_2-1})$ in $(\bbC^2)^{\otimes d_1 + d_2 -1}$.

Write $W_{d_1} \otimes W_{d_2} = \left(W_{d_1 - 1} \otimes x_1 + x_1^{\otimes d_1-1} \otimes y_1 \right) \otimes \left(x_2 \otimes W_{d_2-1} + y_2 \otimes x_2^{\otimes d_2-1}\right)$. We obtain
\[
 \Phi (W_{d_1} \otimes W_{d_2}) = \left( W_{d_1 - 1} \otimes x \otimes x_2^{\otimes d_2-1} \right)+ \left( x_1^{\otimes d_1-1} \otimes x \otimes W_{d_2-1} \right)+ \left( x_1^{\otimes d_1-1} \otimes y \otimes x_2^{\otimes d_2-1}\right).
\]
Notice that $\Phi (W_{d_1} \otimes W_{d_2})$ is a symmetric tensor in $(\bbC^{2})^{\otimes d_1 +d_2 -1}$. After the identification $x_1 \leftrightarrow x_2 \leftrightarrow x$, $y_1 \leftrightarrow y_2 \leftrightarrow y$, we have $\Phi (W_{d_1} \otimes W_{d_2}) = W_{d_1 + d_2 - 1} \in S^{d_1 + d_2 - 1} \bbC^2 \subseteq ( \bbC^2 ) ^{\otimes d_1 + d_2 -1}$. In particular, the tensor rank of $\Phi (W_{d_1} \otimes W_{d_2})$ is $R_{1 \vvirg 1}( W_{d_1+d_2-1}) = d_1 + d_2 -1$. This provides the lower bound $R_{1 \vvirg 1}(W_{d_1} \otimes W_{d_2}) \geq d_1 + d_2 -1$. The general result is obtained by induction on $k$.
\end{proof}

We conclude this section with a multiplicativity result. 

\begin{proposition}\label{prop: multiplicative with W2}
We have $R_{2,d}(W_2 \otimes W_d) =2d$ for every $d\geq 2$.
\end{proposition}

\begin{proof}
By submultiplicativity, $R_{2,d}(W_2\otimes W_d)\le 2d$. For $d = 2$, the statement is true, because $c_{2,2} (W_2 \otimes W_2) = 4$ by Lemma \ref{lemma: cactus of Wd1 tatata Wdk}.

Let $d \geq 3$ and assume by contradiction $R_{2,d}(W_2\otimes W_d)<2d$. Let $S \subseteq \bbP^1 \times \bbP^1$ be a finite set of points evincing $R_{2,d}(W_2\otimes W_d)$.

Since $\deg(S) \le 2d-1 = \dim \vert \Oo _{\bbP^1 \times \bbP^1}(1,d-1) \vert$, we have $\vert \Ii _S(1,d-1) \vert \neq \emptyset$. Moreover, Lemma \ref{lemma: I222 of Z is maximal among schemes spanning T}, implies $\vert \Ii _S(1,d-1) \vert \subseteq \vert \Ii _Z(1,d-1) \vert$. We claim that $\vert \calI_{Z}(1,d-1) \vert = L_2^2 \cdot \vert \calO_{\bbP^1 \times \bbP^1} (1,d-3) \vert$: one inclusion is clear because $Z \subseteq L_2^2$; the other inclusion follows by a dimension count, since $\deg(Z) = 4$, $h^1(\Ii _Z(1,d-1)) =0$, and $h^0(\Oo _{\bbP^1 \times \bbP^1}(1,d-1)) - h^0(\Oo _{\bbP^1 \times \bbP^1}(1,d-3)) =4$.

Let $G \in \vert\calI_S(1,d-1) \vert \subseteq \vert\calI_Z(1,d-1) \vert$ and write $G = L_2^e \cdot F$ for some uniquely determined $e \geq 2$ and $F \in \vert \calO_{\bbP^1 \times \bbP^1} (1,d-1-e) \vert$ with $F \nsupseteq L_2$. Since $S$ is reduced, we have $S\subset G_{\red} = F_{\red} \cdot L_2 \in \vert \calI_S (1,d-1-a)\vert$ for some $a \geq 1$. By Lemma \ref{lemma: I222 of Z is maximal among schemes spanning T}, $ \vert \calI_S (1,d-1-a)\vert \subseteq  \vert \calI_Z(1,d-1-a)\vert$, and therefore we obtain $G_{\red} \in \vert \calI_Z (1,d-1-a)\vert$; by the same argument that we used above, we have $\vert \calI_Z (1,d-1-a)\vert = L_2^2 \cdot \vert \calI_Z (1,d-3-a)\vert$, but this gives a contradiction because $G_{red}$ is reduced but $L_2^2$ is not.
\end{proof}

We mention that since all elements of $S^2 \bbC^2$ are equivalent up to the action of $GL_2$, the $W_2$ in Proposition \ref{prop: multiplicative with W2} can be replaced with $x^2 + y^2$: in particular Proposition \ref{prop: multiplicative with W2} answers, in the partially symmetric setting, the case of Open Problems 16.1 in \cite{cf} where (in the notation of \cite{cf}) $d = k = 2$.

\section{Uniqueness results}\label{section: uniqueness}

In this section, we study uniqueness properties of $Z$ among zero-dimensional schemes whose linear span contains the product of $W$-states.

\begin{theorem}\label{thm: Z unique}
Let $k \geq 1$ and $d_i\ge 3$, $1\le i \le k$. Let $T = W_{d_1}\otimes \cdots \otimes W_{d_k}$. Then $Z_k$ is the only scheme evincing the cactus rank of $T$.
\end{theorem}

\begin{proof}
Let $Z := Z_k$ and let $A$ be a zero-dimensional scheme in $(\bbP^1)^k$ with $\deg(A) = 2^k$ and $T \in \langle \nu_{d_1 \vvirg d_k}(A) \rangle$. Our goal is to prove that $A =Z$. Let $Q_i \in \vert \calO_{(\bbP^1)^k}(0 \vvirg 0,2,0 \vvirg 0) \vert$ be as in Section \ref{section: Preliminaries} (Notation).

The theorem is true for $k=1$ by Sylvester's Theorem (see \cite{cs,bgi,ik}).  Assume $k\ge 2$ and suppose $A \neq Z$. Since $Z = \bigcap_{i} Q_i$, we may assume $A \not \subseteq Q_k$. 

Since $Z \subseteq Q_k \in \vert \calI_Z(0 \vvirg 0,2)$, by Lemma \ref{lemma: 5.1 plus}, we have $h^1 (\calI_{\Res_{Q_k}(A)}(d_1 \vvirg d_{k-1},d_k-2)) > 0$ and therefore $h^1 (\calI_{\Res_{Q_k}(A)}(1 \vvirg 1)) > 0$. This provides $h^1 (\calI_{A}(1 \vvirg 1)) > 0$ and since $\deg(A) = 2^k = h^0(\calO_{(\bbP^1)^k}(1 \vvirg 1)) = 2^k$ we obtain $|\Ii _A(1,\dots ,1)| \ne \emptyset$. 

On the other hand, by Lemma \ref{lemma: I222 of Z is maximal among schemes spanning T}, we have $\vert \Ii _A(1,\dots ,1) \vert \subseteq  \vert \Ii _Z(1,\dots ,1) \vert = \emptyset$. This provides a contradiction and concludes the proof.
\end{proof}

In the case $k = 2$ and $d_i = 3$, we have the following

\begin{proposition}\label{corr1}
Let $T= W_3\otimes W_3 \in S^3 \bbC^2 \otimes S^3 \bbC^2 $. Let $A\subset \PP^1\times \PP^1$ be a zero-dimensional scheme such that $T \in \langle \nu _{3,3}(A)\rangle$ and $\deg (A)\le 5$. Then $Z \subseteq A$.
\end{proposition}

\begin{proof}
 \setcounter{claimCount}{0}

We may assume $\deg(A) = 5$, because $c_{3,3}(T) = 4$ and if $\deg(A) = 4$, then $A = Z$ by Theorem \ref{thm: Z unique}. Suppose $A \nsupseteq Z$, which implies that $A$ is minimal, namely that there is no $A' \subsetneq A$ such that $T \in  \langle \nu_{3,3}(A')\rangle$, because if this was the case, then $\deg(A') = 4$ and $A' = Z$.

Since both $A$ and $Z$ are minimal, Lemma \ref{lemma: union minimals have h^1} implies $h^1(\Ii _{A \cup Z}(3,3)) >0$. 

Since $Z \neq A$ and $Z = Q_1 \cap Q_2$, we may assume $A \not\subseteq Q_2$. 

\begin{claim}\label{blabla1}
We have $A \subseteq Q_1$.
\end{claim}
\begin{quote}
\begin{proof}[Proof of Claim \ref{blabla1}]
Since $A \not\subseteq Q_2$, Lemma \ref{lemma: 5.1 plus} provides $h^1( \calI_{\Res_{Q_2}(Z \cup A)} (3,1) ) > 0$, and since $\Res_{Q_2}(Z \cup A) \subseteq A$, we have $h^1( \calI_{A} (3,1) ) > 0$ and therefore $h^1(\calI_A(2,1)) > 0$. Since $\deg(A) = 5$ and $h^1(\calI_A(2,1)) > 0$, the exact sequence 
\[
0 \to \calI_A(2,1) \to \calO_{\bbP^1 \times \bbP^1}(2,1) \to \calO_{\bbP^1 \times \bbP^1}(2,1) \vert_A \to 0 
\]
provides $h^0(\calI_A(2,1)) = 2$. By Lemma \ref{lemma: I222 of Z is maximal among schemes spanning T}, we have $ \vert \calI_A ( 2,1)\vert \subseteq \vert \calI_Z ( 2,1) \vert$. Since $h^1 (\calI_Z ( 2,1)) = 0$ and $\deg(Z) = 4$, we have $h^0(\calI_Z ( 2,1)) = 2$ so that $\vert \calI_A ( 2,1)\vert = \vert \calI_Z ( 2,1) \vert$; finally, since $Z \subseteq Q_1$, we have $\vert \calI_A ( 2,1) \vert = \vert \calI_Z ( 2,1) \vert = Q_1 \cdot \vert \calO_{\bbP^1 \times \bbP^1}(0,1) \vert$. By definition, $A$ is contained in the base locus of $\vert \calI_A ( 2,1) \vert$, and since $\vert \calO_{\bbP^1 \times \bbP^1}(0,1) \vert$ has no base locus, we deduce that $A \subseteq Q_1$.
\end{proof}
\end{quote}

We have $5 = \deg (A) = \deg(A \cap L_1) + \deg( \Res_{L_1}(A))$. Set $e = \deg(A \cap L_1)$: since $Q_1 = L_1^2$, we have $ \Res_{L_1}(A) \subseteq A \cap L_1$, so that we obtain $2e \geq 5$ and clearly $e < 5 = \deg(A)$. This shows $3 \leq e \leq 4$.

 \begin{claim}\label{blabla2}
We have $h^1(\calI_{\Res_{L_1}(A \cup Z)} (2,3)) = 0$.
\end{claim}
\begin{quote}
\begin{proof}[Proof of Claim \ref{blabla2}]
From the exact sequence
\[
0 \to \calI_{\Res_{L_1}(A \cup Z)} (2,3) \to \calO_{\bbP^1 \times \bbP^1}(2,3) \to \calO_{\bbP^1 \times \bbP^1} (2,3) \vert_{\Res_{L_1}(A \cup Z)} \to 0,
\]
we have that $h^1(\calI_{\Res_{L_1}(A \cup Z)} (2,3)) = 0$ is equivalent to the surjectivity of the restriction map $H^0 (\calO_{\bbP^1 \times \bbP^1} (2,3)) \to H^0 (\calO_{\bbP^1 \times \bbP^1} (2,3) \vert_{\Res_{L_1}(A \cup Z)})$. This restriction map is the composition 
\[
H^0 (\calO_{\bbP^1 \times \bbP^1} (2,3)) \to H^0 (\calO_{\bbP^1 \times \bbP^1} (2,3) \vert_{L_1}) \to H^0 (\calO_{\bbP^1 \times \bbP^1} (2,3) \vert_{\Res_{L_1}(A \cup Z)}): 
\]
the first one is surjective because $h^1(\calO_{\bbP^1 \times \bbP^1}(1,3)) = 0$; the second one is surjective because $\deg(\Res_{L_1}(A \cup Z)) \leq 4$ and $\calO_{\bbP^1 \times \bbP^1} (2,3) \vert_{L_1} \simeq \calO_{\bbP^1}(3)$. 
\end{proof}
\end{quote}
Since $h^1(\calI_{\Res_{L_1}(A \cup Z)}(2,3)) = 0$, we deduce that $\Res_{L_1}(A) \cap \Res_{L_1}(Z) \neq \emptyset$: indeed if $\Res_{L_1}(A) \cap \Res_{L_1}(Z) = \emptyset$, Lemma \ref{lemma: 5.1 plus} would imply $A \cup Z \subseteq L_1$, which is false. In particular, $o \in \Res_{L_1}(A)$ so $o$ appears with multiplicity at least $2$ in $A$, and $o \in A \cap L_1$.

Since $h^1(\Ii _{A \cup Z}(3,3)) >0$, the residual exact sequence of $A \cup Z$ with respect to $L_1$ gives $h^1( \calI_{(A \cup Z) \cap L_1 , L_1} (3,3)) > 0$. From this, we deduce $\deg ( (A \cup Z) \cap L_1) \geq 5$ because $L_1 \simeq \bbP^1$ and $h^0( \calO_{\bbP^1 \times \bbP^1} (2,3) \vert _{L_1}) = h^0(\calO_{\bbP^1} (3)) = 4$. Since $ 5 \leq \deg ( (A \cup Z) \cap L_1) \leq \deg (A \cap L_1) + \deg(Z \cap L_1) - \deg(A \cap Z \cap L_1) \leq e + 2 - 1$, we obtain $e = 4$ and $\deg(A \cap Z \cap L_1) = 1$ so $o$ appears with multiplicity $1$ in $A \cap L_1$. 

 Write $A = A_1 \sqcup A_2$, where $A_1$ is a subscheme of degree $3$ on $L_1 \smallsetminus \{o \}$ and $A_2$ is a scheme of degree $2$ supported at $\{o\}$. Let $A_3$ be the zero-dimensional scheme of degree $4$ in $L_1$ supported at $o$. Then $\langle \nu_{3,3} (A_1) \rangle = \langle \nu_{3,3} (A_3) \rangle$ because they are two subschemes of degree $4$ in $L_1$, and $\nu_{3,3}(L_1)$ is a rational normal curve of degree $3$. Therefore
\[
 \langle \nu_{3,3} (A) \rangle = \langle \nu_{3,3} (A_1)  \rangle +  \langle \nu_{3,3} (A_2)  \rangle  = \langle \nu_{3,3} (A_3)  \rangle +  \langle \nu_{3,3} (A_2)  \rangle  = \langle \nu_{3,3} (A_2 \cup A_3)  \rangle.
\]
Let $\tilde{A} = A_2 \cup A_3$. Then $\deg(\tilde{A}) = \deg(A_2) + \deg(A_3) - \deg(A_2 \cap A_3)$. We have $A_2 \cap A_3 = \{o\}$ because $A_2 \not\subseteq L_1$, so $\deg(\tilde{A}) = 5$. Moreover $Z \not\subseteq \tilde{A}$ because $\Res_{L_1}(A) = \{o\}$ whereas $\Res_{L_1}(A) = Z_1 \otimes \{o_1\}$. Therefore $\tilde{A}$ satisfies the same hypothesis as $A$, and in particular we have $h^1( \calI_{(\tilde{A} \cup Z) \cap L_1 , L_1} (3,3)) > 0$. But $(\tilde{A} \cup Z) \cap L_1 = A_3$ and we have $h^1( \calI_{A_3 , L_1} (3,3)) = 0$ because $A_3$ is a scheme of degree $3$ on $\bbP^1$.
\end{proof}

The result of Proposition \ref{corr1} is sharp because of the following remark:

\begin{remark}\label{corr2}
Fix an integer $k\geq 2$ and integers $d_i\ge 3$, $1\le i\le k$. For any $i\in \{1,\dots ,k\}$ there is a zero-dimensional scheme $A\subset (\PP^1)^k$ with $Z_k \not\subseteq A$ such that $\deg (A)=d_i 2^{k-1}$ and $T = W_{d_1} \ootimes W_{d_k} \in \langle \nu _{d_1,\dots ,d_k}(A)\rangle$. Up to a permutation of the factors, we show that this is true for $i = 1$. Since $R_{d_1} (W_{d_1}) = d_1$, there exists a set $S$ contained in the rational normal curve of degree $d_1$
made of $d_1$ distinct points such that $W_{d_1} \in \langle \nu_{d_1}(S) \rangle$. Take $A = S \times Z_1 \ttimes Z_1$. Then $\deg(A) = d_1 2^{k-1}$, $Z_1 \ttimes Z_1 \not \subseteq A$ and $ T \in \langle \nu_{d_1 \vvirg d_k}(A) \rangle$.

In particular, for $k=2$, $d_1 = d_2 = 3$, we obtain $\deg(A) = 6$, showing that Proposition \ref{corr1} is sharp.
\end{remark}

Let $X\subset \PP^N$ be an irreducible and nondegenerate variety. For every $p \in \PP^N$ let $\Ss (p,X)$ denote the set of all subsets of $X$ evincing the $X$-rank of $p$; more precisely, if $r = R_X(p)$
\[
 \calS(p,X) := \{  (x_1 \vvirg x_r) \in X^{(r)} : p \in \langle x_1 \vvirg x_r \rangle\}
\]
where $X^{(r)} = X^{\times r} / \frakS_r$ denotes the $r$-th symmetric power of the variety $X$. The set $\calS(p,X)$ is constructible and we can study its irreducible components and their dimension. Similarly, let $\calZ(p,X)$ be the set of all zero-dimensional schemes evincing the $X$-cactus rank of $p$.

\begin{proposition}\label{inf2} Fix integers $k$ and $d_1 \vvirg d_k \geq 2$. Let $\scrV_{d_1, \ldots ,d_k}$ be the Segre-Veronese variety.  Let $\Gamma$ be any irreducible component of $\Ss (T, \scrV_{d_1 \vvirg d_k})$. Then $\dim ( \Gamma ) \geq k$ and for every $p=(p_1,\dots ,p_k)\in \bbP^1 \ttimes \bbP^1$ with $p_i \ne [x_i],[y_i]$ there exists $A\in \Gamma$ with $p\in A$.
\end{proposition}

\begin{proof}
The group $SL_2 ^{(1)} \ttimes SL_2^{(k)}$ naturally acts $S^{d_1}\bbC^2 \ootimes S^{d_k} \bbC^2$ and on its projectivization, where $SL_2^{(i)}$ is identified with the group of $2 \times 2$ matrices with determinant $1$ acting on $\bbC^2$ with the given basis $x_i,y_i$. Partially symmetric rank is invariant under the action of the group. For every $i$, let $\Delta_i \subseteq SL_2^{(i)}$ be the subgroup of diagonal matrices, that stabilizes the points $[x_i],[y_i]$ on $\bbP^1$ and therefore the element $[W_{d_i}] \in \bbP S^{d_i} \bbC^2$; for every $i$, we have $\dim ( \Delta_i )= 1$. Therefore $G:= \Delta_1 \ttimes \Delta_k$ stabilizes $T$, and acts on $\calS(T,V_{d_1 \vvirg d_k})$; moreover, since $G$ is connected, it acts on every irreducible component of $\calS(T,\nu_{d_1 \vvirg d_k})$ and in particular on $\Gamma$. The group $G$ has a unique open orbit on $\bbP^1 \ttimes \bbP^1$, that is $U = \{ (p_1\vvirg p_k) : p_i \neq [x_i],[y_i] \} = ( \bbP^1 \ttimes \bbP^1 ) \setminus L$ where $L$ is the divisor defined by $L = L_1 \cdots L_k \in \calO_{\bbP^1 \ttimes \bbP^1} (1 \vvirg 1)$.

Fix $S \in \Gamma$. Notice that $S \cap U \neq \emptyset$: indeed if $S \subseteq L$, then $\Res_L( S)  = \emptyset$ and since $\Res_L(Z) = \{ o \}$, we have $h^1( \calI_{\Res_{L}(Z \cup S)} (d_1 - 1 \vvirg d_k-1)) = 0$. By Lemma \ref{lemma: 5.1 plus} we obtain $Z \subseteq L$ which is false. Therefore $S \cap U \neq \emptyset$ and this implies that no element of $G$ stabilizes $S$. The orbit of $S$ under the action of $G$ is contained in $\Gamma$, and since $\dim G = k$, we conclude.
\end{proof}

\section{Conclusions}

We conclude with a brief discussion on some of the many problems that remain open and on the limits of our techniques.

Proposition \ref{prop: upper bound Wd1 ot Wd2} suggests that the bound of Theorem \ref{bo1} is not tight, but obtaining better bounds using similar techniques seems a hard task. On the other hand, the upper bound of Theorem \ref{thm: rank upper bound W3 tatata W3} is tight for $k=1,2$ and there is numerical evidence that it is tight for $k=3$. We pose the problem of determining expressions of the form \eqref{eqn: R333 expression for W3 tatata W3} in the case of higher $d_i$'s, which would lead to better upper bounds in the general case. 

As far as lower bounds are concerned, it is clear that the result of Proposition \ref{lower1} is far from sharp (even asymptotically). The techniques of Proposition \ref{fingercrossed} may provide results for higher degrees and higher number of factors, but one would need a general framework to avoid the case by case discussion.

In the more general setting of Segre-Veronese varieties, we pose following technical questions, which can be addressed using techniques similar to the ones used in this work.

\begin{question}\label{question: final questions}
Let $T\in \mathbb{P}(S^{d_1} \bbC^{n_1+1} \ootimes S^{d_k} \bbC^{n_k+1})=\mathbb{P}^N$ with
$N = -1 + \prod _{i=1}^{k} \binom{n_i+d_i}{n_i}$ and indicate with $\scrV_{d_1, \ldots , d_k}^{n_1, \ldots , n_k}$ the Segre-Veronese embedding of $\mathbb{P}^{n_1}\ttimes  \mathbb{P}^{n_k}$ of multidegree $(d_1 \vvirg d_k)$ in $\bbP^N$. Let $T \in \bbP^N$.

\begin{enumerate}
\item What is the minimum integer $m_0$ such that there exists a zero-dimensional scheme (resp. a finite set) $A$ in $\scrV_{d_1, \ldots , d_k}^{n_1, \ldots , n_k}$, with $\deg (A)=m_0$, such that $T\in \langle \nu _{d_1 \vvirg d_k}(A)\rangle$ and $A \nsupseteq B$ for any $B\in \Zz (T, \scrV_{d_1, \ldots , d_k}^{n_1, \ldots , n_k})$ (resp. for any $B\in \Ss(T,\scrV_{d_1, \ldots , d_k}^{n_1, \ldots , n_k})$)?
\item What is the minimum integer $m_0$ such that for every $m$ with $ m_0 \leq m \leq N+1$ there exists a zero-dimensional scheme (resp. a finite set) $A_{m}$ such that $\deg(A_m) = m$, $T\in \langle \nu_{d_1\vvirg  d_k}(A_m)\rangle$ and there is no $A' \subseteq A_m$ such that $T\notin \langle \nu _{d_1 \vvirg d_k}(A')\rangle$? 
\item What are necessary conditions (or sufficient conditions) for $T$ to satisfy $|\Zz _X(T)|=1$ or $|\Ss _X(T)|=1$?
\end{enumerate}
\end{question}

\section{Appendix} \label{appendix}
\subsection{Considerations on Sylvester's theorem}

We briefly present Sylvester's Theorem for binary forms (\cite{cs}, \cite{bgi}, \cite[\S 1.3]{ik}). It is a classical result that completely describes the behavior of rank and cactus rank in the setting of complex homogeneous polynomials of degree $d$ in two variables, namely elements of $S^d \bbC^2$.

\begin{remark}\label{remark: sylvester thm}
Fix $d \geq 1$. All integers $r = 1 \vvirg d$ occur as $\scrV_d$-rank of some element in $S^d \bbC^2$ and all integers $r = 1 \vvirg \lceil (d+1)/2\rceil$ occur as $\scrV_d$-cactus rank. For elements of rank $r \leq \lceil(d+1)/2 \rceil$, $\scrV_d$-cactus rank, $\scrV_d$-border rank and $\scrV_d$-border cactus rank coincide (see \cite{bgi}) and the generic $\scrV_d$-rank is $\lceil (d+1)/2\rceil$. For every $f \in S^d \bbC^2$, there is a  unique scheme $A_f \subseteq \bbP^1$ that evinces the cactus rank of $f$, unless $d$ is even and $c_d(f) = \frac{d}{2} +1$; in this case $R_d(f) = \frac{d}{2} +1$ as well, and rank and cactus rank are evinced by infinitely many sets of distinct points (an instance of this phenomenon is in the case $f = x^{d/2}y^{d/2}$). If $A_f$ is reduced, then $A_f$ evinces the rank as well, and $R_d ( f) = c_d(f)$. If $A_f$ is not reduced, then $R_d(f) = d+2 - c_d(f)$. 

Explicitly, for every $r = 1 \vvirg \lceil (d+1)/2\rceil$, the polynomial $ f_r = \sum_1^r (x+ j y)^d$  satisfies $c_d(f_r) = R_d(f) = r$ and $A_{f_r} = \{ [x+ j y] : j = 1 \vvirg r\} \subseteq \bbP^1$. For every $r = \lceil (d+1)/2\rceil +1 \vvirg d$, the polynomial $g_r = x^{d-r+1} y^{r-1}$ satisfies $c_d(g_r) = r$, $R_d(g_r) = d+2 - r$ and $A_{g_r}$ is the zero-dimensional scheme supported at $[x]$ of degree $r$. In particular $R_d(W_d) = d$, $c_d(W_d) = 2$ and $A_{W_d} = Z_1$.

For every $r\in \{1,\dots ,d\}$, the set
$S[r]:=\{f \in \PP^d :  R_{d}(f)=r\}
$
is a constructible set, in the sense of \cite[\S2.C]{Mum:ComplProjVars}. Moreover, it is irreducible of dimension $2d+2-2r$ (see also \cite{bhmt}). 
\end{remark}

We can regard Remark \ref{remark: sylvester thm} as a particular case of the following remark, that applies to every projective variety.

\begin{remark}\label{remark: basics on rhoX}

 Let $X \subseteq \bbP^N$ be a projective nondegenerate variety. Define 
 \begin{align*}
  \rho_X^\circ &= \max\{ r : \text{ for any set $S \subseteq X$  of $r$ distinct points, $\dim \langle S \rangle = r-1$  } \}, \\
  \rho_X  &= \max\{ r :  \text{for any zero-dimensional scheme $A \subseteq X$ of degree $r$, $\dim \langle A \rangle = r-1$  } \}.
 \end{align*}
  
Equivalently, $\rho_X$ is the maximum integer such that any zero-dimensional scheme $A \subseteq X$ with $\deg(A) \leq \rho_X$ satisfies $h^1(\calI_A(1)) = 0$ and similarly in the case of a set of distinct points and $\rho_X^\circ$. Clearly $2 \leq \rho_X \leq \rho_X^\circ$. Moreover $\rho_X = \rho_X^\circ = 2$ whenever $X$ contains a line.

Let $p \in \bbP^N$ and suppose that $A \subseteq X$ is a zero-dimensional scheme such that $p \in \langle A \rangle$ and there does not exist $A' \subseteq A$ such that $p \in \langle A' \rangle$. We have $c_X(p) \leq \deg(A)$.

If $A$ is not reduced, then $R_X(p) \geq \rho_X + 1 - \deg(A)$. Indeed, if $S$ is a set of distinct points evincing $R_X(p)$, we have $h^1(\calI_{A \cup S}(1)) > 0$, so $\deg(A \cup S) \geq  \rho_X +1$ by Lemma \ref{lemma: union minimals have h^1} and therefore $\deg(A) + \deg(S) \geq \rho_X+1$. Hence, $R_X(p) \geq \rho_X + 1 - \deg(A)$. In addition, if $\deg(A) \leq \frac{1}{2}(\rho_X + 1)$, then $c_X(p) = \deg(A)$. Indeed, if $c_X(p) < \deg(A)$, let $B$ be a zero-dimensional scheme evincing $c_X(p)$. We have $\deg(B) < \deg(A)$; on the other hand $h^1(\calI_{A \cup B} (1) ) >0$ by Lemma \ref{lemma: union minimals have h^1}, so $\deg(A \cup B) \geq \rho_X + 1$. We conclude $\deg(A) + \deg(B) \geq \rho_X + 1$ but $\deg(A) + \deg(B) < 2 \deg(A)\leq = \rho_X + 1$, which is a contradiction.  Similarly, if $A \leq \frac{1}{2}\rho_X$, then $A$ is the unique zero-dimensional scheme evincing $c_X(p)$ because if $B \neq A$ evinces $c_X(p)$, we obtain a contradiction as above.

Similar considerations hold for sets of distinct points $S$ and the relation between $\deg(S)$ and $\rho_X^\circ$. 
\end{remark}
We point out that for every $d \geq 1$, $\rho_{\scrV_d} = \rho_{\scrV_d}^\circ = d+1$; in particular $\rho_{\scrV_d}$ does not depend on $n$ and in the case $n=2$ Remark \ref{remark: basics on rhoX} reduces to Remark \ref{remark: sylvester thm}. Similarly for every $d_1 \vvirg d_k$, $\rho_{\scrV_{d_1 \vvirg d_k}} = \rho_{\scrV_{d_1 \vvirg d_k}}^\circ = \min \{d_i +1 : 1 \leq i \leq k \}$; again, $\rho_{\scrV_{d_1 \vvirg d_k}}$ does not depend on the dimension of the factors of $\scrV_{d_1 \vvirg d_k}$.

\subsection{Flattenings}\label{appendix: flattenings}
A classical approach to determine lower bounds on rank and border rank is via flattening maps. Let $V$ be a vector space. Given two vector spaces $E,F$, a flattening of $V$ is a linear map $Flat_{E,F} : V \to \Hom(E,F)$, that associates to every element $T \in V$ a linear map $T_{E,F}: E \to F$.  In particular, if $X \subseteq \bbP V$ is a projective variety, write $r_0 = \max\{ \rank (p_{E,F}) : p \in X \}$ (here $\rank$ denotes the rank of the linear map $p_{E,F}$). Then, for every $T \in V$, we have (see e.g. \cite[Prop. 4.1.1]{landott} and also  \cite[Lemma 18]{cjz})
\begin{equation}\label{eqn: flattening lower bound}
\uR_X(T) \geq \frac{1}{r_0} \rank(T_{E,F}). 
\end{equation}

In the setting of partially symmetric tensors, a particular class of flattenings arises naturally via tensor contraction. Fix $d_1 \vvirg d_k$ and $n_1 \vvirg n_k$. For every choice of $e_1 \vvirg e_k$ with $0 \leq e_i \leq d_i$, we define the flattening map sending $T \in S^{d_1} \bbC^{n_1} \ootimes S^{d_k} \bbC^{n_k}$ to the linear map
\[
 T_{e_1 \vvirg e_k} : S^{e_1} \bbC^{n_1 *} \ootimes S^{e_k} \bbC^{n_k*} \to S^{d_1 - e_1} \bbC^{n_1} \ootimes S^{d_k - e_k} \bbC^{n_k},
\]
given canonically by contraction. More explicitly, if $S^d \bbC^n$ is identified with the space of homogeneous polynomials of degree $d$ in $n$ variables, then $S^e \bbC^{n*}$ can be interpreted as the space of differential operators of order $e$ with constant coefficients, and the map above is given simply by applying a differential operator to every factor of $T$. In this case, the denominator in \eqref{eqn: flattening lower bound} is $r_0 = 1$, so we have $R_{d_1 \vvirg d_k} (T) \geq \uR_{d_1 \vvirg d_k}(T) \geq \rank (T_{e_1 \vvirg e_k})$ for any choice of $e_1 \vvirg e_k$.

Thm. 4 in \cite{gal} shows that for a large family of flattening maps, including the ones just presented, the border rank lower bound of equation \eqref{eqn: flattening lower bound} holds for $X$-border cactus rank as well. Therefore we obtain $c_{d_1 \vvirg d_k} (T) \geq \uc_{d_1 \vvirg d_k}(T) \geq \rank (T_{e_1 \vvirg e_k})$.

Moreover, if $T^{(1)}$ has a flattening $T^{(1)}_{E_1,F_1}$ and $T^{(2)}$ has a flattening $T^{(2)}_{E_2 ,F_2}$ then the map $T^{(1)}_{E_1,F_1} \boxtimes T^{(2)}_{E_2,F_2} : E_1 \otimes E_2 \to F_1 \otimes F_2$ is a flattening of $T^{(1)} \otimes T^{(2)}$ and we have $\rank (T^{(1)}_{E_1,F_1} \boxtimes T^{(2)}_{E_2,F_2}) = \rank(T^{(1)}_{E_1,F_1}) \rank(T^{(2)}_{E_2,F_2})$ (see also \cite[Section 4]{cjz}).

This guarantees that lower bounds on border cactus rank obtained via the flattening maps presented above are multiplicative. In particular, in the symmetric case, we have the following.

\begin{remark}
 Fix $d_1 \vvirg d_k, n_1 \vvirg n_k$ and consider $f_i \in S^{d_i} \bbC^{n_i}$.  If, for $i=1 \vvirg k$, $f_i$ is generic in $\sigma_{r_i}( \nu_{d_i}(\mathbb{P}^{n_i}))$ with $r_i \leq \binom{\lfloor d_i/2\rfloor + n_i -1}{d_i}$, then $R_{d_1 \vvirg d_k} (f_1 \ootimes f_k) = \prod R_{d_i} (f_i) = \prod r_i$. This follows from the fact that generic elements of $\sigma_{r_i}( \nu_{d_i}(\mathbb{P}^{n_i}))$ have rank equal to $r_i$ and from \cite[Prop. 3.12]{ik} which guarantees that for a generic element of $\sigma_{r_i}( \nu_{d_i}(\mathbb{P}^{n_i}))$ with $r_i \leq  \binom{\lfloor d_i/2\rfloor + n_i -1}{d_i}$, then the flattening $( f_i )_{\lfloor d_i/2 \rfloor}$ has rank exactly $r_i$. Indeed, in this range, one verifies $R_{d_i}(f_i) = c_{d_i}(f_i)$ and there exists a unique zero-dimensional scheme (which is in fact a set of distinct points) $S_i \subseteq \bbP^1$ such that $\deg(S_i) = r_i$ and $f_i \in \langle \nu_{d_i}(S_i) \rangle$ (see e.g. Remark \ref{remark: basics on rhoX}).
\end{remark}

The previous observation motivates the following questions, which gives a more precise version of the problem posed in part (iii) of Question \ref{question: final questions}
\begin{question}
 Fix $d_1 \vvirg d_k, n_1 \vvirg n_k$ and $f_i \in S^{d_i} \bbC^{n_i}$ with cactus rank $r_i \leq \lceil d_i/2 \rceil$. Let $A_i\subset \PP^{n_i-1}$ be the unique zero-dimensional scheme of degree $r_i$ evincing the cactus rank of $f_i$. Is $A_1\times \cdots \times A_k$ the only zero-dimensional scheme evincing the cactus rank of $f_1 \ootimes f_k$? Is this true in the very special case where $n_i = 2$, $d_i$ is odd and each $f_i$ is a general element of $S^{d_i}(\CC ^2)$ (so $r_i = \lceil d_i/2 \rceil$)?
\end{question}

\subsection{Minimally spanning schemes and linear independence}\label{section: minimally spanning}

In this section, we present an example which shows that a minimally spanning scheme (in the sense of inclusion) is not necessarily linearly independent. More formally, we describe an example of the following situation: $X \subseteq \bbP^N$ is an irreducible smooth variety and $p \in \bbP^N$ is a point such that there exists a zero-dimensional scheme $B \subseteq X$ such that $B$ minimally spans $p$ (namely $p \in \langle B \rangle$ and there is no $B' \subsetneq B$ such that $p \in \langle B' \rangle$) and $B$ is linearly dependent (in the sense that $h^1(\calI_B(1)) > 0$). In particular, this shows that the hypothesis $h^1(\calI_B(1)) = 0$ in Lemma \ref{lemma: 5.1 plus} is necessary and does not follow from the other hypothesis of the lemma.

We point out that in our example $B$ does not evince the cactus rank of the point $p$. In particular the minimality of $B$ is only with respect to inclusion (as in the hypothesis of Lemma \ref{lemma: 5.1 plus}), not with respect of degree.

For the theory underlying this example, we refer to \cite[Ex. IV.3.3.3 and Sec. IV.4]{hart} and \cite[Sec. III.3]{shaf}.

\begin{example}\label{example: minimally spanning with h1}
 Let $C$ be a curve of genus $1$ and let $\calL$ be a line bundle of degree $N+1$ on $C$. Then $\calL$ provides an embedding $\phi_\calL : C \to \bbP^N \simeq \bbP H^0(C,\calL)^*$ as a normal elliptic curve in $\bbP^N$ with $\deg(\phi_L(C)) = N+1$: let $X = \phi_\calL (C)$. There are exactly $(N+1)^2$ points $ q \in C$ such that the divisor $(N+1)q$ is an element of $\vert \calL(1) \vert$; equivalently, this means that the zero-dimensional scheme $B_1 = \phi_L( (N+1)q) \subseteq X$ is given by the intersection between $X$ and a hyperplane $H \subseteq \bbP^N$.
 
 Observe that $\langle B_1 \rangle = H$: since $B_1 \subseteq H$ then $\langle B_1 \rangle \subseteq H$; on the other hand, if $\langle B_1 \rangle = M \subsetneq H$, consider a hyperplane containing $M$ and another point $z \in X$; this hyperplane intersects $X$ in a zero-dimensional scheme of degree at least $\deg( (N+1)q) + \deg(z) = N+1 + 1 = N+2$ in contradiction with Bezout's Theorem since $\deg(X) = N+1$. 
 
Now let $B = \phi_L((N+2)q)$. Then $\langle B \rangle = \bbP^N$: indeed $H \subseteq \langle B \rangle$ and if equality holds then $H$ intersects $X$ in the scheme $(N+2)q$, whose degree is $N+2$, again in contradiction with Bezout's Theorem.

Let $p$ be any point in $\bbP^N$ such that $ p\notin H$. Then $p \in \langle B \rangle$. Since $B$ is connected and supported on $X$, every proper subscheme of $B$ is contained in $B_1$, so that if $B' \subsetneq B$, then $\langle B' \rangle \subseteq \langle B_1 \rangle = H$, and $p \notin \langle B' \rangle $. This shows that $B$ minimally spans $p$. On the other hand, $h^1(\calI_B(1)) \geq \deg(B) - (\dim \langle B \rangle +1) = 1$. In fact $h^1(\calI_{B_1}(1)) = 1$ already.

More explicitly, when $N = 2$, let $q$ be one of the $(2+1)^2 = 9$ flexes on a smooth plane cubic, that are the $9$ points given by the intersection between the plane cubic and its Hessian, that is the curve cut out by the determinant of the matrix of second order partial derivatives. Then, the scheme $3q$ supported at $q$ is contained in (and in fact spans) the tangent line at $q$. If $p$ is a point not lying on the tangent line at $q$, then $p \notin \langle 3q \rangle$ and $p \in \langle 4q\rangle$. On the other hand $\deg(4q) = 4$, and $\dim \langle 4q \rangle = 2$, so $h^1(\calI_{4q}(1)) = 4 - (2+1) = 1$.

We observe that one can construct examples similar to the one above using flexes of higher order to obtain zero-dimensional schemes $B$ minimally spanning a point $p$ with $h^1(\calI_B(1))$ arbitrarily large. It suffices to consider a curve $X \subseteq \bbP^N$ with the property that $X \cap H$ is a connected zero-dimensional scheme $B_1$ of high degree, so that $h^1(\calI_{B_1}(1)) = m$ is arbitrarily large and then repeat the same construction as above. Curves of high degree with these properties exist for every $m$ and every $N$. When $N= 2$, these examples can be constructed as soon as the degree of the curve is at least $m+2$.
\end{example}

\bibliographystyle{amsalpha}
\bibliography{bibfileBBCG-Wstates}

\newcommand{\etalchar}[1]{$^{#1}$}
\providecommand{\bysame}{\leavevmode\hbox to3em{\hrulefill}\thinspace}
\providecommand{\MR}{\relax\ifhmode\unskip\space\fi MR }
\providecommand{\MRhref}[2]{%
  \href{http://www.ams.org/mathscinet-getitem?mr=#1}{#2}
}
\providecommand{\href}[2]{#2}
\begin{thebibliography}{BDHM17}

\bibitem[AB13]{ab3}
H.~Abo and M.~C. Brambilla, \emph{{On the dimensions of secant varieties of
  {S}egre-{V}eronese varieties}}, Ann. Mat. Pur. ed Appl. \textbf{192} (2013),
  no.~1, 61--92.

\bibitem[BB12]{bb1}
E.~Ballico and A.~Bernardi, \emph{{Decomposition of homogeneous polynomials
  with low rank}}, Mathematische Zeitschrift \textbf{271} (2012), no.~3-4,
  1141--1149.

\bibitem[BB13]{bb2}
\bysame, \emph{{Stratification of the fourth secant variety of {V}eronese
  varieties via the symmetric rank}}, Adv. Pure and Appl. Math. \textbf{4}
  (2013), no.~2, 215--250.

\bibitem[BB14]{BuczBucz:SecantsVeroneseCataSmoothGorSch}
W.~Buczy\'{n}ska and J.~Buczy\'{n}ski, \emph{{Secant varieties to high degree
  Veronese reembeddings, catalecticant matrices and smoothable {G}orenstein
  schemes}}, J. Alg. Geom. \textbf{23} (2014), no.~1, 63--90.

\bibitem[BBM14]{berbramou}
A.~Bernardi, J.~Brachat, and B.~Mourrain, \emph{{A comparison of different
  notions of ranks of symmetric tensors}}, Lin. Alg. Appl. \textbf{460} (2014),
  205--230.

\bibitem[BC12]{BernCaru:AlgGeomToolsEntanglement}
A.~Bernardi and I.~Carusotto, \emph{{Algebraic geometry tools for the study of
  entanglement: an application to spin squeezed states}}, J. Phys. A
  \textbf{45} (2012), no.~10, 105304, 13.

\bibitem[BCZ17]{BuhChrZui:NondetQCC_CyclicEqGameIMM}
H.~Buhrman, M.~Christandl, and J.~Zuiddam, \emph{{Nondeterministic quantum
  communication complexity: the cyclic equality game and iterated matrix
  multiplication}}, {Proc. of the 2017 ACM Conference on ITCS}, 2017.

\bibitem[BD10]{bd}
K.~Baur and J.~Draisma, \emph{{Secant dimensions of low-dimensional homogeneous
  varieties}}, Adv. in Geometry \textbf{10} (2010), no.~1, 1--29.

\bibitem[BDHM17]{BerDalHauMou:TensorDecompHomCont}
A.~Bernardi, N.~S. Daleo, J.~D. Hauenstein, and B.~Mourrain, \emph{{Tensor
  decomposition and homotopy continuation}}, Diff. Geom. Appl. \textbf{55}
  (2017), 78--105.

\bibitem[BGI11]{bgi}
A.~Bernardi, A.~Gimigliano, and M.~Id\`{a}, \emph{{Computing symmetric rank for
  symmetric tensors}}, J. Symb. Comput. \textbf{46} (2011), no.~1, 34--53.

\bibitem[BHMT17]{bhmt}
J.~Buczy{\'n}ski, K.~Han, M.~Mella, and Z.~Teitler, \emph{{On the locus of
  points of high rank}}, European J. Math. (2017), 1--24.

\bibitem[BLR80]{BinLotRom:ApproxSolutionsBilFormCompProb}
D.~Bini, G.~Lotti, and F.~Romani, \emph{{Approximate solutions for the bilinear
  form computational problem}}, SIAM J. Comput. \textbf{9} (1980), no.~4,
  692--697.

\bibitem[BR13]{br}
A.~Bernardi and K.~Ranestad, \emph{{On the cactus rank of cubic forms}}, J.
  Symb. Comput. \textbf{50} (2013), 291--297.

\bibitem[BT15]{bt}
G.~Blekherman and Z.~Teitler, \emph{{On maximum, typical and generic ranks}},
  Mathematische Annalen \textbf{362} (2015), no.~3-4, 1021--1031.

\bibitem[CCD{\etalchar{+}}10]{CheChiDuaJiWin:TensorRankStocEntCatalysisMultPureStates}
L.~Chen, E.~Chitambar, R.~Duan, Z.~Ji, and A.~Winter, \emph{{Tensor rank and
  stochastic entanglement catalysis for multipartite pure states}}, Phys. Rev.
  Let. \textbf{105} (2010), no.~20, 200501.

\bibitem[CF18]{cf}
L.~Chen and S.~Friedland, \emph{{The tensor rank of tensor product of two
  three-qubit {W} states is eight}}, Lin. Alg. Appl. \textbf{543} (2018),
  1--16.

\bibitem[CGG05]{cgg}
M.~V. Catalisano, A.~V. Geramita, and A.~Gimigliano, \emph{{Higher secant
  varieties of {S}egre-{V}eronese varieties}}, {Projective varieties with
  unexpected properties}, Walter de Gruyter GmbH \& Co. KG, Berlin, 2005,
  pp.~81--107.

\bibitem[CGJ18]{ChrGesJen:BorderRankNonMult}
M.~Christandl, F.~Gesmundo, and A.~K. Jensen, \emph{{Border rank is not
  multiplicative under the tensor product}}, arXiv: 1801.04852 (2018).

\bibitem[CGO14]{cgo:four}
E.~Carlini, N.~Grieve, and L.~Oeding, \emph{{Four lectures on secant
  varieties}}, {Connections between algebra, combinatorics, and geometry},
  Springer, 2014, pp.~101--146.

\bibitem[CJZ18]{cjz}
M.~Christandl, A.~K. Jensen, and J.~Zuiddam, \emph{{Tensor rank is not
  multiplicative under the tensor product}}, Lin. Alg. Appl. \textbf{543}
  (2018), 125--139.

\bibitem[CLO07]{CLO}
D.~Cox, J.~Little, and D.~O'shea, \emph{{Ideals, {V}arieties, and {A}lgorithms:
  An {I}ntroduction to {C}omputational {A}lgebraic {G}eometry and {C}ommutative
  {A}lgebra}}, {{Undergraduate Texts in Mathematics}}, Springer-Verlag New
  York, Inc., 2007.

\bibitem[CS11]{cs}
G.~Comas and M.~Seiguer, \emph{On the rank of a binary form}, Found. Comp.
  Math. \textbf{11} (2011), no.~1, 65--78.

\bibitem[CU13]{CohUma:FastMatMultCohConf}
H.~Cohn and C.~Umans, \emph{{Fast matrix multiplication using coherent
  configurations}}, {Proc. of the 24th Ann. ACM-SIAM Symp. on Disc. Alg.},
  SIAM, 2013, pp.~1074--1087.

\bibitem[CW90]{CopperWinog:MatrixMultiplicationArithmeticProgressions}
D.~Coppersmith and S.~Winograd, \emph{{Matrix multiplication via arithmetic
  progressions}}, J. Symb. Comput. \textbf{9} (1990), no.~3, 251--280.

\bibitem[DVC00]{MR1804183}
W.~D\"{u}r, G.~Vidal, and J.~I. Cirac, \emph{{Three qubits can be entangled in
  two inequivalent ways}}, Phys. Rev. A (3) \textbf{62} (2000), no.~6, 062314,
  12.

\bibitem[EH00]{EisHar:GeometrySchemes}
D.~Eisenbud and J.~Harris, \emph{{The geometry of schemes}}, {Graduate Texts in
  Mathematics}, vol. 197, Springer-Verlag, New York, 2000.

\bibitem[EH16]{EisHar:3264}
\bysame, \emph{{3264 and {A}ll {T}hat - {A} {S}econd {C}ourse in {A}lgebraic
  {G}eometry}}, Cambridge University Press, Cambridge, 2016.

\bibitem[FL07]{FortLo:RandomBipartiteEntFromW}
B.~Fortescue and H.-K. Lo, \emph{{Random bipartite entanglement from {W and
  W-like states}}}, Phys. Rev. Let. \textbf{98} (2007), no.~26, 260501.

\bibitem[FL08]{FortLo:RandomPartEntDistMultipartyStates}
B.~Fortescue and H.-K. Lo, \emph{{Random-party entanglement distillation in
  multiparty states}}, Phys. Rev. A \textbf{78} (2008), no.~1, 012348.

\bibitem[Ga{\l}17]{gal}
M.~Ga{\l}\k{a}zka, \emph{{Vector bundles give equations of cactus varieties}},
  Lin. Alg. Appl. \textbf{521} (2017), 254--262.

\bibitem[Ger96]{Gera:InvSysFatPts}
A.~V. Geramita, \emph{{Inverse systems of fat points: {W}aring{\rq}s problem,
  secant varieties of {V}eronese varieties and parameter spaces for
  {G}orenstein ideals}}, {The curves seminar at {Q}ueen{\rq}s}, vol.~10, 1996,
  pp.~2--114.

\bibitem[GH94]{GrifHar:PrinciplesAlgebraicGeometry}
P.~A. Griffiths and J.~Harris, \emph{{Principles of algebraic geometry}},
  {Wiley Classics Library}, John Wiley \& Sons Inc., New York, 1994, Reprint of
  the 1978 original.

\bibitem[Har77]{hart}
R.~Hartshorne, \emph{{Algebraic geometry}}, vol.~52, Springer-Verlag, New York,
  1977, Graduate Texts in Mathematics.

\bibitem[HLT12]{HolLuqThi:GeomDescrEntangAuxVars}
F.~Holweck, J.-G. Luque, and J.-Y. Thibon, \emph{{Geometric descriptions of
  entangled states by auxiliary varieties}}, J. Math. Phys. \textbf{53} (2012),
  no.~10, 102203.

\bibitem[IK99]{ik}
A.~Iarrobino and V.~Kanev, \emph{{Power sums, {G}orenstein algebras, and
  determinantal loci}}, {Lecture Notes in Mathematics}, vol. 1721,
  Springer-Verlag, Berlin, 1999, Appendix C by Iarrobino and Steven L. Kleiman.

\bibitem[{Le }14]{LeGall:PowersTensorsFastMatrixMult}
F.~{Le Gall}, \emph{{Powers of tensors and fast matrix multiplication}}, {Proc.
  of the 39th Int. Symp. on Symb. and Alg. Comp.}, ACM, 2014, pp.~296--303.

\bibitem[LO13]{landott}
J.~M. Landsberg and G.~Ottaviani, \emph{{Equations for secant varieties of
  {V}eronese and other varieties}}, Ann. Mat. Pura Appl. \textbf{192} (2013),
  no.~4, 569--606.

\bibitem[LS88]{LovSaks:LatticesMoebFuncsCommCompl}
L.~Lov\`{a}sz and M.~Saks, \emph{{Lattices, M\"{o}bius functions and
  communications complexity}}, {29th Ann. Symp. on Found. of Comp. Sc.}, IEEE,
  1988, pp.~81--90.

\bibitem[MS82]{MehSch:LasVegasBetter}
K.~Mehlhorn and E.~M. Schmidt, \emph{{Las {V}egas is better than determinism in
  {VLSI} and distributed computing}}, {Proc. 14th annual ACM Symp. on Th. of
  Comp.}, ACM, 1982, pp.~330--337.

\bibitem[Mum95]{Mum:ComplProjVars}
D.~Mumford, \emph{{Algebraic geometry. {I}: Complex projective varieties}},
  {Classics in Mathematics}, Springer-Verlag, Berlin, 1995, Reprint of the 1976
  edition in Grundlehren der mathematischen Wissenschaften, vol. 221.

\bibitem[RS11]{ransch}
K.~Ranestad and F.-O. Schreyer, \emph{{On the rank of a symmetric form}}, J.
  Algebra \textbf{346} (2011), 340--342.

\bibitem[Sha94]{shaf}
I.~R. Shafarevich, \emph{{Basic algebraic geometry. 1 - Varieties in projective
  space}}, second ed., Springer-Verlag, Berlin, 1994.

\bibitem[Sto10]{Stoth:ComplexityMatrixMultiplication}
A.~J. Stothers, \emph{{On the complexity of matrix multiplication}}, Ph.D.
  thesis, 2010.

\bibitem[Str83]{Strassen:RankOptimalComputationGenericTensors}
V.~Strassen, \emph{{Rank and optimal computation of generic tensors}}, Lin.
  Alg. Appl. \textbf{52/53} (1983), 645--685.

\bibitem[Syl52]{Sylvester1852}
J.J. Sylvester, \emph{{On the principles of the calculus of forms}}, Cambridge
  and Dublin Mathematical Journal (1852), 52--97.

\bibitem[VC17]{VraChr:EntanglDistillGHZ}
P.~Vrana and M.~Christandl, \emph{{Entanglement Distillation from
  Greenberger--Horne--Zeilinger Shares}}, Comm. Math. Phys. \textbf{352}
  (2017), no.~2, 621--627.

\bibitem[Wil12]{Williams:MultMatricesFasterCW}
V.~V. Williams, \emph{{Multiplying matrices faster than
  {Coppersmith-Winograd}}}, {Proc. of the 44th annual ACM Symp. on Th. of
  Comp.}, ACM, 2012, pp.~887--898.

\bibitem[YCGD10]{YCGD10}
N.~Yu, E.~Chitambar, C.~Guo, and R.~Duan, \emph{{Tensor rank of the tripartite
  state $\vert W \rangle^{\otimes n}$}}, Phys. Rev. A \textbf{81} (2010),
  no.~1, 014301.

\bibitem[Zui17]{Zuid:NoteGapRankBorderRank}
J.~Zuiddam, \emph{{A note on the gap between rank and border rank}}, Lin. Alg.
  Appl. \textbf{525} (2017), 33--44.

\end{thebibliography}

\end{document}